\documentclass[11pt]{article}

\pdfoutput=1

\usepackage[T1]{fontenc}
\usepackage[italian,english]{babel}
\usepackage[final]{microtype}
\usepackage{enumitem}
\usepackage{amsthm, amsmath, amssymb}
\usepackage{tikz-cd}
\usepackage{bm} 
\usepackage{cancel}
\usepackage{verbatim} 
\usepackage[pdfusetitle,pdfborder={0 0 .5}]{hyperref}
\usepackage{cleveref}

\newcommand\extrafootertext[1]{%
	\bgroup
	\renewcommand\thefootnote{\fnsymbol{footnote}}%
	\renewcommand\thempfootnote{\fnsymbol{mpfootnote}}%
	\footnotetext[0]{#1}%
	\egroup
}

\theoremstyle{plain}
\newtheorem{theorem}{Theorem}[section]
\newtheorem{lemma}[theorem]{Lemma}
\newtheorem{proposition}[theorem]{Proposition}
\newtheorem{corollary}[theorem]{Corollary}

\theoremstyle{definition}
\newtheorem{definition}[theorem]{Definition}
\newtheorem{remark}[theorem]{Remark}
\newtheorem*{remark*}{Remark}
\newtheorem{example}[theorem]{Example}
\newtheorem{question}[theorem]{Question}
\newtheorem{conjecture}[theorem]{Conjecture}

\newcommand*{\ie}{{\it i.e.}}
\newcommand*{\eg}{{\it e.g.}}
\newcommand*{\R}{\mathbb{R}} 
\newcommand*{\Z}{\mathbb{Z}} 

\newcommand*{\abs}[1]{\left\lvert#1\right\rvert} 

\newcommand*{\SV}[1]{\left\lVert#1\right\rVert} 

\newcommand*{\Dav}[1]{M(#1)} 

\newcommand*{\Link}[2]{\mathrm{Link}_{#2}({#1})} 

\newcommand*{\ldom}{\preceq}
\newcommand*{\gdom}{\succeq}
\newcommand*{\minit}{basic}

\title{Simplicial maps between spheres and Davis' manifolds with positive simplicial volume}
\author{Francesco Milizia\\
	{\small \textit{Scuola Normale Superiore, Palazzo della Carovana}}\\
	{\small \textit{Piazza dei Cavalieri, 7, 56126 Pisa, IT}}\\
	{\small \texttt{francesco.milizia@sns.it}}
	\\
}

\date{}

\begin{document}
	\maketitle
	
	\begin{abstract}
		We study the simplicial volume of manifolds obtained from Davis' reflection group trick, the goal being characterizing those having positive simplicial volume.
		In particular, we focus on checking whether manifolds in this class with nonzero Euler characteristic have positive simplicial volume (Gromov asked whether this holds in general for aspherical manifolds).
		This leads to a combinatorial problem about triangulations of spheres: we define a partial order on the set of triangulations --- the relation being the existence of a nonzero-degree simplicial map between two triangulations --- and the problem is to find the minimal elements of a specific subposet.
		We solve explicitly the case of triangulations of the two-dimensional sphere, and then perform an extensive analysis, with the help of computer searches, of the three-dimensional case.
		Moreover, we present a connection of this problem with the theory of graph minors.
	\end{abstract}
	
	\section{Introduction}
	\extrafootertext{\scriptsize{The author has been supported by the INdAM -- GNSAGA project CUP E55F22000270001.}}
	The work presented in this paper is motivated by the following question formulated by Mikhail Gromov.
	\begin{question}[{\cite[page 232]{Gromov93}}]\label{conj:sv_chi}
		Let $M$ be a closed (\ie, compact and without boundary) connected aspherical manifold.
		Suppose that the simplicial volume of $M$ vanishes.
		Then, does the Euler characteristic $\chi(M)$ vanish?
	\end{question}
	This question, which tries to establish a relation between two important invariants of manifolds, has stimulated a lot of work in the community studying simplicial volume, but a complete answer to it still seems to be widely out of reach.
	The recent paper \cite{LMR22} contains a discussion of various versions of \Cref{conj:sv_chi}, a list of special cases in which the conjectured implication is known to hold, and a collection of strategies towards positive and negative answers to it.
	
	The simplicial volume of a manifold $M$ is a nonnegative real number, denoted by $\SV{M}$; we recall its definition in \Cref{sec:pre}.
	The simplicial volume is invariant under homotopy equivalences between closed manifolds; at the same time, it can say a lot about the possible geometries a given manifold can be endowed with: for instance, it gives a lower bound for the minimal volume of a smooth manifold \cite[p.~220]{Gromov82}.
	Also the Euler characteristic provides similar connections between topology and geometry, motivating the search of precise relations between the two invariants.
	The aim of this paper is to study the simplicial volume of a specific class of manifolds, obtained from a construction due to Michael W.~Davis, and in particular to check whether these manifolds satisfy the implication stated in \Cref{conj:sv_chi}.
	
	Asphericity is a necessary assumption in \Cref{conj:sv_chi}: the simplest example of a nonaspherical manifold, the 2-dimensional sphere $S^2$, has vanishing simplicial volume, but $\chi(S^2) = 2$.
	Davis' construction, explained in \Cref{sec:davis_construction}, is an abundant source of aspherical manifolds.
	Over the last decades, its versatility has led to the discovery of closed aspherical manifolds enjoying a variety of exotic properties: having a fundamental group which is not residually finite \cite{Mess1990}, with unsolvable word problem \cite[Theorem 13.4]{Dav2002}, or admitting locally CAT(0) metrics but no smooth Riemannan metrics of nonpositive sectional curvatures \cite{DJL2012}, just to list some examples.
	
	In this paper, we consider only the simplest version (as described, for instance, in \cite{Dav2002}) of the construction, which is a particular case of the more general \emph{reflection group trick} introduced by Davis in \cite{Dav83}.
	It takes as input a triangulation $T$ of the $(n-1)$-dimensional sphere, and gives back a closed manifold of dimension $n$, that we denote with $\Dav{T}$.
	Hereafter, by triangulation we mean a simplicial complex whose geometric realization is homeomorphic to a given topological space.
	If the triangulation is \emph{flag} (see \Cref{sec:pre} for the definition of this combinatorial property), then the resulting manifold is aspherical.
	
	Most of the effort in this paper is put into studying flag triangulations of spheres and simplicial maps between them.
	A central notion in this work is the following.
	\begin{definition}\label{def:dominance}
		Let $T_1$ and $T_2$ be two triangulation of the $(n-1)$-dimensional sphere $S^{n-1}$.
		We say that $T_2$ dominates $T_1$, and write $T_1 \ldom T_2$, if there is a simplicial map from $T_2$ to $T_1$ of nonzero degree.
	\end{definition}
	The dominance relation gives a partial order to the set of (isomorphism classes of) triangulations of $S^n$.
	We prove in \Cref{sec:davis_construction} that, if $T_1 \ldom T_2$, then $\SV{\Dav{T_1}} \le \SV{\Dav{T_2}}$.
	In particular, if the simplicial volume of $\Dav{T}$ is positive, for a certain triangulation $T$, then the same holds for every triangulation that dominates $T$.
	This observation is the starting point for all the subsequent work in this paper.
	
	Consider the set of triangulations $T$ of $S^{n-1}$, considered up to simplicial isomorphism, such that $\SV{\Dav{T}} > 0$.
	As a consequence of the observation just stated, this set is completely determined by its minimal elements with respect to the dominance relation $\ldom$.
	The main problem addressed in this paper is finding these minimal elements.
	
	The case $n = 2$ (triangulations of $S^1$, giving rise to surfaces) is easy to understand: $\SV{\Dav{T}} > 0$ if and only if $T$ has at least $5$ vertices; the poset of triangulations is linear; there is only one minimal element in the subposet of triangulations giving positive simplicial volume --- the pentagon.
	
	\subsection{The three-dimensional case}
	When $n = 3$, the poset becomes much more complicated.
	Nonetheless, we prove in \Cref{sec:3dim} that the subposet of triangulations giving positive simplicial volume has --- again --- only one minimal element, which is a flag triangulation of $S^2$ having $9$ vertices.
	
	\begin{theorem}[{\Cref{thm:2dim}}]\label{thm:2dim_intro}
		There exists a triangulation of $S^2$, that we call $T_9$ (see \Cref{fig:T9}), which is flag, has $9$ vertices, and is such that the following conditions on a triangulation $T$ of $S^2$ are equivalent:
		\begin{enumerate}[label=(\roman*)]
			\item \label{it:2_pos_intro}
			The three-dimensional manifold $\Dav{T}$ has positive simplicial volume;
			\item \label{it:2_t9_intro}
			There is a simplicial map $T\to T_9$ of degree $1$;
			\item \label{it:2_dom_intro}
			$T_9 \ldom T$.
		\end{enumerate}
	\end{theorem}
	
	There are several reasons that make the case $n = 3$ tractable, compared to higher-dimensional cases that seem much more complicated.
	The first reason is that a 3-dimensional aspherical manifold has positive simplicial volume if and only if its fundamental group is relatively hyperbolic (with respect to a collection of proper subgroups).
	This fact, specific to dimension 3, and false in higher dimensions, follows from the geometrization theory of 3-manifolds.
	Moreover, there is an algorithmic procedure to check whether a manifold of the form $\Dav{T}$ has a relatively hyperbolic fundamental group (see \Cref{rmk:rel_hyp}).
	These facts alone don't imply anything about the set of minimal elements of the subposet mentioned above; but, at least, they already allow us to determine easily if a given triangulation belongs to the subposet or not.
	
	Our solution of the case $n = 3$ is rather elementary and does not use the geometrization theorem, but it is interesting to notice, by inspecting the proof of \Cref{thm:2dim_intro} presented in \Cref{sec:3dim}, how the steps used to manipulate triangulations correspond --- at the level of the associated Davis' manifolds --- to cutting operations along spheres and tori, \ie, to the operations one has to perform to decompose a manifold into its prime pieces and then into geometric pieces.
	
	The second, maybe even stronger, reason that explains the tractability of the case $n = 3$ comes from graph theory, more precisely from the theory of graph minors developed by Neil Robertson and Paul D.~Seymour.
	A minor of a graph is another graph obtained from the first one by erasing some edges, erasing some vertices and collapsing some edges.
	In \Cref{sec:minors} we prove the following:
	\begin{proposition}[{\Cref{prop:minor_s2}}]
		\label{prop:minor_s2_intro}
		Let $T_1$ and $T_2$ be two triangulations of $S^2$.
		Suppose that the 1-skeleton of $T_1$ is a minor of the 1-skeleton of $T_2$.
		Then, $T_1 \ldom T_2$.
		More precisely, there is a simplicial map $f:T_2 \to T_1$ with $\abs{\deg(f)} = 1$.
	\end{proposition}
	This opens the door to the powerful results of graph minor theory, that imply at once not only that it is possible to decide algorithmically whether a triangulation of $S^2$ belongs to the ``positive simplicial volume'' subposet (see \Cref{cor:dim2_alg}), but also that this subposet has a finite number of minimal elements.
	One caveat: these results don't actually give an algorithm, indications on how to find the minimal elements, or information about how many there are.
	Our proof of \Cref{thm:2dim_intro} provides concretely the unique minimal element.
	
	Among the consequences of \Cref{prop:minor_s2_intro}, we also obtain the following result about simplicial maps between triangulated spheres:
	\begin{proposition}[{\Cref{cor:dim2_deg}}]
		Let $T$ be a triangulation of $S^2$.
		There is a natural number $d_T$ such that, if $S$ is any triangulation of $S^2$ that dominates $T$, then it does so via a simplicial map $f$ with $\abs{\deg(f)} \le d_T$.
	\end{proposition}
	
	We do not know if this result can be generalized to higher dimensions, and whether the number $d_T$ can always be taken equal to $1$.
	
	\subsection{The four-dimensional case}
	In higher dimensions we do not have at our disposal the rich theory just described for the 3-dimensional case.
	We don't know if the ``positive simplicial volume'' subposet has a finite number of minimal elements, and in general we are not able to tell, given a triangulation, whether it belongs to the subposet.
	
	In \Cref{sec:4dim} we focus on the case $n = 4$, and consider a different subposet: the one consisting of flag triangulations $T$ of $S^3$ such that $\chi(\Dav{T}) \neq 0$.
	The advantage in considering this subposet is that it is easy to tell whether a triangulation belongs to it or not.
	To test \Cref{conj:sv_chi}, we should check that every triangulation in this ``nonzero Euler characteristic'' subposet gives a manifold with positive simplicial volume; again, it is enough to perform this check on the minimal elements of the subposet.
	Such minimal elements constitute the ``basic'' examples on which \Cref{conj:sv_chi} can be tested, and we find that there are at least two of them:
	\begin{proposition}
		There exist two distinct triangulations of $S^3$, that we call $T_{10}$ and $T_{12}$, with the following properties:
		\begin{enumerate}[label=(\arabic*)]
			\item \label{it:flag_intro} They are flag;
			\item \label{it:euler_intro} Their corresponding Davis' manifolds have nonzero Euler characteristic;
			\item They do not dominate (in the sense of \Cref{def:dominance}) other triangulations satisfying properties \ref{it:flag_intro} and \ref{it:euler_intro}.
		\end{enumerate}
	\end{proposition}
	The triangulations $T_{10}$ and $T_{12}$ have 10 and 12 vertices respectively.
	The manifold $\Dav{T_{10}}$ is homeomorphic to the product of two surfaces of genus 5; therefore, it has positive simplicial volume.
	We have not been able to reach the same conclusion for $\Dav{T_{12}}$.
	We leave it as an open question (which is a particular instance of \Cref{conj:sv_chi}):
	\begin{question}
		Does $\Dav{T_{12}}$ have positive simplicial volume?
	\end{question}
	
	To search for other minimal triangulations, it is possible to perform random ``explorations'' in the poset of flag triangulations of $S^3$.
	The details of such a procedure, and the results obtained by the author by implementing it on a computer (a C++ implementation is available in the GitHub repository \cite{simplicial-spheres}), are discussed in \Cref{sec:4dim}.
	The algorithm produces flag triangulations satisfying the following condition: every $1$-simplex is contained in a \emph{square}, \ie, a full subcomplex homeomorphic to $S^1$ induced by $4$ vertices.
	This is justified by the following proposition:
	\begin{proposition}[{\Cref{prop:many_squares}}]
		Let $T$ be a flag triangulation of $S^3$, with $\chi(\Dav{T}) \neq 0$, that does not dominate any other (nonisomorphic) triangulation with the same properties.
		Then, either $T$ is isomorphic to $T_{10}$, or every edge of $T$ is contained in a square.
	\end{proposition}
	
	The number of \emph{pairwise nonisomorphic} triangulations that have been constructed by running the algorithm is reported in \Cref{table:summary}.
	\begin{table}[ht]
		\centering
		\begin{tabular}{|c|c|c|}
			\hline
			N. of triangulations & dominating $T_{10}$ & dominating $T_{12}$ \\
			\hline
			4\,075\,183 & 4\,074\,851 ($\approx 99.992\,\%$) & 3\,995\,971 ($\approx 98.056\,\%$)\\
			\hline
		\end{tabular}
		\caption{Summary of the results of obtained by running the implemented algorithm, in which triangulations with up to $\approx 50$ vertices have been constructed and checked.
		All of them --- except the one with $8$ vertices --- dominate $T_{10}$ or $T_{12}$ (or both).}
		\label{table:summary}
	\end{table}
	All the triangulations counted in the summary of \Cref{table:summary} are flag and satisfy the condition on the squares.
	The leftmost column includes, among the others, the (unique) flag triangulation of $S^3$ with $8$ vertices, whose corresponding Davis' manifold has vanishing Euler characteristic.
	Of course, this triangulation dominates neither $T_{10}$ nor $T_{12}$.
	Every other triangulation that has been constructed gives nonvanishing Euler characteristic; in fact, according to a conjecture of Lutz and Nevo (see \Cref{conj:collapses}), the triangulation with $8$ vertices should be the only flag triangulation with $\chi(\Dav{T}) = 0$ that satisfies the condition on the squares mentioned above.
	
	To determine whether a triangulation $T$ dominates $T_{10}$ or $T_{12}$, one could list all simplicial maps from $T$ to $T_{10}$ or $T_{12}$, and then compute their degree; however, this procedure is prohibitively long even for relatively small triangulations with $\approx 20$ vertices.
	To overcome this difficulty, we introduce in Subsection \ref{ssec:local} a special kind of simplicial maps between triangulations, having a predetermined structure, and that always have degree $\pm 1$.
	The implemented algorithm, once a triangulation $T$ has been constructed, checks whether there is a simplicial map from $T$ to $T_{10}$ or $T_{12}$ of this special kind --- this turns out to be feasible even for $T$ having hundreds of vertices.
	In this way, some \emph{false negatives} could occur, \ie, a triangulation $T$ might dominate $T_{10}$ or $T_{12}$ but \emph{not} with a map of that special kind, and the dominance relation \emph{would not} be detected.
	In practice, this didn't turn out to be a problem: apart from the one with $8$ vertices, every other constructed triangulation has been found to dominate either $T_{10}$, $T_{12}$, or both, via a special map (the figures in the last two columns of \Cref{table:summary} are to be intended in this way). Indeed, the vast majority, according to the experiments, dominates $T_{10}$.

	\subsection*{Structure of the paper}
	In \Cref{sec:pre}, we recall the definition of simplicial volume, and the basic notation and constructions concerning simplicial complexes.
	In \Cref{sec:3dim}, we characterize, using the dominance relation, the triangulations $T$ of $S^2$ such that $\SV{\Dav{T}} > 0$.
	In \Cref{sec:4dim}, we discuss the case where $T$ is a flag triangulation of $S^3$, our strategy to find the minimal elements of the poset of flag triangulations with $\chi(\Dav{T}) \neq 0$, and the results of an implementation of this strategy on a computer.
	A C++ implementation is available in a GitHub repository \cite{simplicial-spheres}.
	In \Cref{sec:minors}, we establish the connection between our problem and graph minor theory, and discuss its theoretical and algorithmic consequences.
	Finally, in \Cref{sec:questions}, we consider several questions related to the problems studied in this paper.
	
	\subsection*{Acknowledgements}
	I am grateful to my PhD advisor Roberto Frigerio, for his support throughout the development of this project.
    Moreover, I wish to thank several people for helpful discussions and comments about this work and related topics: Christos Athanasiadis, Giuseppe Bargagnati, Federica Bertolotti, Pietro Capovilla, Jean-Fran\c{c}ois Lafont, Kevin Li, Marco Moraschini, Stefano Riolo, Lorenzo Venturello.

	\section{Preliminaries}\label{sec:pre}
	
	\subsection{Simplicial volume}
	Let $M$ be a closed oriented $n$-dimensional manifold.
	Its singular homology in degree $n$, with integer coefficients, is isomorphic to $\Z$, with the unit $1 \in \Z$ corresponding to a class in $H_n(M;\Z)$ called the fundamental class of $M$, denoted by $[M]$.
	The inclusion $\Z \hookrightarrow \R$ induces a change-of-coefficients map in homology.
	The image of $[M]$ under this change of coefficients is denoted by $[M]_\R$, is an element of $H_n(M;\R) \cong \R$, and is called the real fundamental class of $M$.
	
	The simplicial volume of $M$ is defined \cite{Gromov82} as the infimum of the $\ell^1$-norm of representatives of $[M]_\R$ (which are finite combinations of singular simplices, with real coefficients):
	\[ \SV{M} = \inf\left\{ \sum_{i=1}^k \abs{a_i} \colon \sum_{i=1}^k a_i\sigma_i \text{ singular cycle representing } [M]_\R \right\}.\]
	We now list some classical facts about simplicial volume that will be used throughout the paper.
	In particular, the first fact listed below, despite being very elementary, will be fundamental in our approach for characterizing Davis' manifolds with positive simplicial volume.
	\begin{itemize}
		\item Let $f:M \to N$ be a continuous map between oriented closed manifolds.
		Then, $\abs{\deg(f)} \cdot \SV{N} \le \SV{M}$.
		This is proved by considering the push-forward via $f$ of representatives of $[M]_\R$.
		In particular, if $M$ admits a self-map $f:M \to M$ with $\abs{\deg(f)} \ge 2$, then $\SV{M} = 0$.
		For instance, $\SV{S^n} = 0$ for every $n \ge 1$.
		\item $\SV{M}\SV{N} \le \SV{M \times N} \le C\SV{M}\SV{N}$, where $C$ is a constant depending only on $\dim(M\times N)$.
		The second inequality is straightforward, because given representatives for $[M]_\R$ and $[N]_\R$ one can easily construct a representative for $[M\times N]_\R$.
		The other inequality is proved by using bounded cohomology \cite{Gromov82, Frigerio2017}.
		\item Oriented closed connected hyperbolic manifolds have positive simplicial volume \cite{Thurston80, Gromov82}.
		In particular, oriented closed surfaces of genus $\ge 2$ have positive simplicial volume.
	\end{itemize}
	The last item has many generalizations, the general theme being that ``enough'' negative curvature implies positivity of simplicial volume.
	For instance, a closed Riemannian manifold with strictly negative sectional curvatures has positive simplicial volume \cite{IY82}.
	Several results of this kind are listed in \cite[Example 3.1]{LMR22}.
	We will use the following one, in which the negative-curvature condition is replaced by an assumption on the fundamental group.
	\begin{proposition}[{\cite[Proposition 1.6]{BH2013}}]\label{prop:relhyp_sv}
		Let $M$ be an $n$-dimensional closed connected manifold ($n \ge 2$) which is aspherical and whose fundamental group is relatively hyperbolic with respect to a finite collection of proper subgroups.
		Then the simplicial volume of $M$ is positive.
	\end{proposition}
	We do not recall here the definition of relatively hyperbolic groups.
	For an account on this notion, we refer the reader to \cite{Osin2006}.
	For our purposes, it is enough to know that there exists an effective algorithm to check, given a simplicial complex $T$, whether the fundamental group of the Davis' manifold $\Dav{T}$ is relatively hyperbolic (see \Cref{prop:relhyp_coxeter}).
	
	\subsection{Simplicial complexes}
	Let $T$ be a finite simplicial complex with vertex set $V(T)$.
	Formally, $T$ is a family of subsets of $V(T)$,
	but we will often confuse $T$ with its geometric realization, which is a topological space.
	We always assume that $\emptyset \in T$ and $\{v\}\in T$ for every $v \in V(T)$.
	
	We recall that the link of a simplex $\sigma \in T$ is the subcomplex
	\[\Link{\sigma}{T} = \{\tau \in T \colon \tau \cap \sigma= \emptyset\text{ and }\tau \cup \sigma \in T\}.\]
	The notion of link is defined also for faces of a \emph{cubical complex}.
	If $F$ is a face of a cubical complex $X$, then $\Link{F}{X}$ is a simplicial complex, whose simplices are in 1-to-1 correspondence with the faces of $X$ containing $F$, with the same containment relation.
	
	Given two simplicial complexes $T_1$ and $T_2$, with disjoint vertex sets, their join is the simplicial complex defined as
	\[T_1 \star T_2 = \{\sigma_1 \cup \sigma_2 \colon \sigma_1 \in T_1\text{ and }\sigma_2 \in T_2\}.\]
	If $T_1$ and $T_2$ are triangulations of $S^{n-1}$ and $S^{m-1}$, respectively, then $T_1 \star T_2$ is a triangulation of $S^{n+m-1}$.
	
	A \emph{full} subcomplex $S \subseteq T$ is a subcomplex with the property that a simplex $\sigma \in T$ belongs to $S$ if and only if their vertices do (so, $S$ is determined univocally by its vertex set $V(S)$).
	
	\begin{definition}\label{def:flag}
		A simplicial complex
		$T$ is \emph{flag} if, for every set of vertices $\{v_0, \dots, v_k\} \subseteq V(T)$ such that $\{v_i,v_j\} \in T$ for every $(i,j) \in \{0,\dots,k\}^2$, \ie, they are pairwise adjacent, it holds $\{v_0, \dots, v_k\} \in T$, \ie, they are the vertices of a simplex of $T$.
		In other words, $T$ is the maximal simplicial complex with a given $1$-skeleton.
	\end{definition}
	
	As an example, a triangulation of $S^1$ is flag if and only if it has at least $4$ vertices (with $3$ vertices, there would be a missing $2$-simplex).
	
	The property of being flag is preserved by taking links, joins, and full subcomplexes.
	If $T$ is flag, then every link in $T$ is a full subcomplex.
	
	\begin{definition}[Edge subdivision]
		Let $T$ be a simplicial complex and let $\{x,y\} \in T$ be a $1$-simplex of $T$.
		We obtain a new simplicial complex $T'$ as follows: a new vertex $m_{x,y}$ is introduced, and $T'$ is defined as
		\begin{align*}
			T' = \{\sigma \colon \{x,y\}\not\subseteq \sigma \in T\}
			     &\cup \{\sigma\setminus\{y\}\cup\{m_{x,y}\} \colon \{x,y\}\subseteq \sigma \in T\}\\
			     &\cup \{\sigma\setminus\{x\}\cup\{m_{x,y}\} \colon \{x,y\}\subseteq \sigma \in T\}.
		\end{align*}
		This formalizes the insertion of a midpoint on the $1$-simplex $\{x,y\}$, that requires the simplices containing $\{x,y\}$ to be subdivided accordingly.
		We say that $T'$ is obtained from $T$ by subdividing $\{x,y\}$.
	\end{definition}
	The geometric realizations of $T$ and $T'$ are homeomorphic.
	If $T$ is flag, then also $T'$ is flag.
	At the level of the $1$-skeleton, the new vertex $m_{x,y}$ is adjacent to $x$, $y$, and to the common neighbours of $x$ and $y$.
	
	\begin{definition}[Edge collapse]
		Let $T$ be a simplicial complex and let $\{x,y\} \in T$ be a $1$-simplex of $T$.
		We obtain a new simplicial complex $T'$ with $V(T') = V(T) \setminus \{y\}$ and
		\begin{align*}
			T' = \{\sigma \colon y \not\in \sigma \in T\}
			     \cup \{\sigma \setminus \{y\} \cup \{x\} \colon y \in \sigma \in T\}.
		\end{align*}
		We say that $T'$ is obtained from $T$ by collapsing $\{x,y\}$.
	\end{definition}
	It follows from the definitions that the subdivision of an edge $\{x,y\}$ followed by the collapse of $\{x,m_{xy}\}$ (or $\{y,m_{xy}\}$) gives back the initial simplicial complex.
	
	In general, an edge collapse can change the homeomorphism type of a simplicial complex.
	However, we recall the following theorem of Nevo:
	\begin{theorem}[{\cite{Nevo2007}}]
		Given an edge $\{x,y\}$ in a triangulation $T$ of a compact PL (piecewise-linear) manifold\footnote{In the PL setting, a triangulation is a simplicial complex whose geometric realization is PL homeomorphic (not just homeomorphic) to the given PL manifold.} without boundary, its collapse results in a PL-homeomorphic space if and only if it satisfies the following \emph{link condition}:
		\[ \Link{\{x\}}{T} \cap \Link{\{y\}}{T} = \Link{\{x,y\}}{T}. \]
	\end{theorem}
	Note that the link condition is always satisfied if $T$ is flag.
	In particular, if $T$ is a triangulation of $S^{n-1}$ equipped with the standard PL structure,\footnote{If $n \le 4$, every triangulation of $S^{n-1}$ is PL homeomorphic to the standard PL $S^{n-1}$.} then $T'$ is still a triangulation of $S^{n-1}$.

	\begin{definition}
		Let $T$ be a simplicial complex.
		A \emph{square} in $T$ is a full subcomplex homeomorphic to $S^1$ having $4$ vertices.
	\end{definition}
	If $T$ is flag and an edge $\{x,y\}\in T$ is \emph{not} contained in a square, then the result of the collapse of $\{x,y\}$ is still flag.
	This is not hard to prove, and a detailed proof is provided in \cite{LN2016}.
	In particular, if an edge of a flag triangulation of $S^2$ or $S^3$ is not contained in a square, then the result of its collapse is still a flag triangulation of $S^2$ or $S^3$.

	\section{Davis' construction}\label{sec:davis_construction}
	Starting from a simplicial complex $T$, one can build a cubical complex $\Dav{T}$ with the following procedure introduced by Davis (more details can be found in \cite{Dav2002}).
	The cubical complex $\Dav{T}$ will be a subset of an ambient cube $[-1,1]^{V(T)}$, whose dimension is equal to the number of vertices of $T$.
	Given a face of the ambient cube, consider the subset of $V(T)$ corresponding to the coordinates that are free to vary in the face (while each of the other coordinates is fixed at $-1$ or $+1$); we call this subset the \emph{type} of the face.
	Thus, if $\sigma$ is a subset of $V(T)$, the faces of type $\sigma$ are those of the form
	\begin{align*} \prod_{v \in V(T)} I_v\ \subseteq\ [-1,1]^{V(T)},&& \text{where } I_v = 
	\begin{cases}[-1,1] & \text{if } v \in \sigma\\
	             \{-1\} \text{ or } \{+1\} &\text{if } v\not\in\sigma.\end{cases}\end{align*}
	In particular, there are $2^{\abs{V(T)}-\abs{\sigma}}$ faces of type $\sigma$.
	
	\begin{definition}
		The cubical complex $\Dav{T}$ is the subcomplex of $[-1,1]^{V(T)}$ given by the faces whose types correspond to simplices in $T$.
	\end{definition}
	In particular, the vertices of the ambient cube are always contained in $\Dav{T}$, as their type is the empty simplex $\emptyset$, which always belongs to $T$.
	The cubical complex $\Dav{T}$ is always connected, because the $1$-dimensional faces of the ambient cube correspond to the vertices of $T$, so they are included in $\Dav{T}$.
	
	It readily follows from the construction that if a cube in $\Dav{T}$ has type $\sigma$ then its link in $\Dav{T}$ is isomorphic, as a simplicial complex, to the link of $\sigma$ in $T$.
	In particular, the link of every vertex of $\Dav{T}$ is isomorphic to $T$.
	This easy but fundamental observation has the following consequences:
	\begin{enumerate}[label=(\arabic*)]
		\item If $T$ is homeomorphic to the $(n-1)$-dimensional sphere, then $\Dav{T}$ is a manifold of dimension $n$;
		\item\label{it:cat0} $\Dav{T}$ is aspherical if and only if $T$ is a flag simplicial complex \cite[Corollary 5.4]{Dav2002}.
		In fact, in this case $\Dav{T}$ acquires a locally CAT(0) metric, by Gromov's criterion \cite[\S 4.2.C]{Gromov87}.
	\end{enumerate}
	
	\paragraph{Examples}
	\begin{itemize}
		\item
		If $T$ consists of two points (the $0$-dimensional sphere $S^0$), then $\Dav{T}$ is homeomorphic to $S^1$.
		\item
		If $T$ is an empty triangle (\ie, the boundary of a $2$-simplex, which is not flag), then $\Dav{T}$ is homeomorphic to $S^2$.
		In general, if $T$ is the boundary of a $k$-simplex, then $M(T)$ is homeomorphic to $S^k$.
		\item If $T$ the boundary of a square, then $\Dav{T}$ is homeomorphic to a torus.
		This can be seen either directly by drawing a picture, or as a special case of the next example, or by using the following fact, which is easy to prove: 
		\begin{remark}\label{rmk:join_product}
			$\Dav{T_1\star T_2} = \Dav{T_1} \times \Dav{T_2}$.
		\end{remark}
		\item If $T$ is the boundary of a polygon with $n$ sides, then $\Dav{T}$ is an orientable closed surface of genus $(n-4)\cdot 2^{n-3}+1$.
		This can be proved with an induction argument: adding a vertex in $T$ corresponds to removing a number of disjoint discs from $\Dav{T}$ and taking the double of the resulting surface with boundary.
		Alternatively, one can use \Cref{lemma:chi_formula} below to compute the Euler characteristic of the resulting surface.
	\end{itemize}
	
	\begin{lemma}\label{lemma:chi_formula}
		Let $T$ be a finite simplicial complex of dimension $d$, and for every $i \in \{-1,0,\dots,d\}$ let $f_i$ be the number of $i$-simplices in $T$ (where $f_{-1} = 1$, counting the empty simplex).
		Then the Euler characteristic of $\Dav{T}$ is given by
		\begin{align*}
			\chi(\Dav{T})\ &=\ 2^{f_0-1} \cdot \sum_{i = -1}^d (-1)^{i+1}2^{-i} f_i\\
			&=\ (-1)^{d+1} \cdot 2^{f_0-d-1} \cdot (f_d - 2f_{d-1} + 4f_{d-2} - \dots).
		\end{align*}
	\end{lemma}
	\begin{proof}
		For each $i$-simplex in $T$ there are $2^{f_0-i-1}$ cells in $\Dav{T}$ whose type is equal to the given simplex.
		This contributes with a $(-1)^{i+1}2^{f_0-i-1}$ summand in the computation of the Euler characteristic as the alternating sum of the number of cells in each dimension.
		The result follows.
	\end{proof}
	
	\subsection{The fundamental group of \texorpdfstring{$\bm{\Dav{T}}$}{M(T)}}
	Given a simplicial complex $T$, consider its $1$-skeleton as a Coxeter diagram, and define $\Gamma_T$ as the right-angled Coxeter group associated to it.
	It is not hard to see that the fundamental group of $\Dav{T}$ is isomorphic to a finite-index subgroup of $\Gamma_T$ (see, \eg, \cite{Dav08}).
	More precisely, it is isomorphic to the kernel $G_T < \Gamma_T$ of the homomorphism 
	\[ \Gamma_T \to (\Z/2\Z)^{V(T)} \]
	that sends the generator associated to a vertex $v \in V(T)$ to the corresponding basis element of $(\Z/2\Z)^{V(T)}$.
	The isomorphism $G_T \cong \pi_1(\Dav{T})$ is induced by the natural correspondence between words in the generators of $\Gamma_T$ and paths in the $1$-skeleton of $\Dav{T}$ starting from the vertex whose coordinates are all equal to $+1$.
	
	\begin{proposition}\label{prop:relhyp_coxeter}
		Let $T$ be a finite simplicial complex.
		If $\Gamma_T$ is relatively hyperbolic with respect to a finite collection of proper subgroups, then the same is true for $G_T \cong \pi_1(\Dav{T})$, and vice-versa.
	\end{proposition}
	\begin{proof}
		The property of being relatively hyperbolic with respect to a finite collection of proper subgroups is invariant under quasi-isometries \cite{Drutu2009}; in particular, it is inherited by (and from) finite-index subgroups.
		%
	\end{proof}
	
	In \cite{BHS17}, Behrstock, Hagen and Sisto exhibit a polynomial-time algorithm that decides whether a right-angled Coxeter group is relatively hyperbolic.
	By \Cref{prop:relhyp_coxeter}, we can use this algorithm to decide whether $\pi_1(\Dav{T})$ is relatively hyperbolic (recall that this implies $\SV{\Dav{T}} > 0$ by \Cref{prop:relhyp_sv}).
	This algorithm will be used, in \Cref{sec:3dim}, in order to establish that $\SV{\Dav{T}} > 0$ for a specific triangulation $T$ of $S^2$.
	
	\subsection{Subcomplexes}
	Let $S$ be a subcomplex of $T$.
	Associated to it, we have the cubical complex $M(S)$, the right-angled Coxeter group $\Gamma_S$ and its finite-index subgroup $G_S\cong\pi_1(\Dav{S})$.
	
	The cubical complex $\Dav{S}$ sits naturally inside $\Dav{T}$, by fixing at $+1$ the coordinates corresponding to the vertices in $V(T)\setminus V(S)$.
	By fixing arbitrarily the coordinates associated to vertices not in $S$, one obtains $2^{\abs{V(T)}-\abs{V(S)}}$ disjoint copies of $\Dav{S}$ in $\Dav{T}$.
	
	Suppose now that, whenever two vertices of $S$ are adjacent in $T$, they are also adjacent in $S$ (this is the case if, for instance, $S$ is a full subcomplex).
	Then, the Coxeter group $\Gamma_S$ is a subgroup of $\Gamma_T$, and $\Gamma_S \cap G_T = G_S \cong \pi_1(\Dav{S})$.
	In particular, the inclusion of $\Dav{S}$ in $\Dav{T}$ is $\pi_1$-injective.

	\section{Maps between Davis' cubical complexes}
	Let $T_1$ and $T_2$ be finite simplicial complexes, and let $f:T_1 \to T_2$ be a simplicial map.
	We think of $f$ as a map from $V(T_1)$ to $V(T_2)$, with the property that if some vertices are contained in a simplex of $T_1$ then their images are contained in a simplex of $T_2$, but also as a (continuous) map between the geometric realizations.
	\begin{remark}
		If $T_2$ is flag, and $f:V(T_1)\to V(T_2)$ sends each pair of adjacent vertices to adjacent vertices (or to a single vertex), then $f$ is simplicial.
	\end{remark}
	The upshot of this section is the following proposition and its immediate consequence about simplicial volumes, \Cref{cor:possv}.
	The reader can skip the technicalities of the constructions and proofs in this section, and just keep in mind these results (and \Cref{def:dominance}, which introduces the main notion investigated in subsequent sections) for later.
	
	\begin{proposition}\label{prop:nonzero_induced}
		Let $f:T_1 \to T_2$ be a simplicial map between simplicial complexes homeomorphic to $S^n$, and suppose that $f$ has nonzero degree.
		Then, there is a continuous map from $\Dav{T_1}$ to $\Dav{T_2}$ of nonzero degree.
	\end{proposition}
	The conclusion makes sense because $\Dav{T_1}$ and $\Dav{T_2}$ are topological $(n+1)$-manifolds, and are orientable by the following lemma, which should be well known, but we include a proof below, in order to introduce explicitly representatives of the fundamental classes.
	
	\begin{lemma}\label{lemma:orientations}
		Let $T$ be a simplicial complex homeomorphic to $S^n$.
		Then $\Dav{T}$ is a closed, connected and orientable $(n+1)$-manifold.
	\end{lemma}
	
	Recall that, given two simplicial complexes $T_1$ and $T_2$, both homeomorphic to $S^n$,
	we say that $T_2$ dominates $T_1$, and write $T_1 \ldom T_2$, if there is a simplicial map from $T_2$ to $T_1$ of nonzero degree.
	
	\begin{corollary}\label{cor:possv}
		Let $T_1$ and $T_2$ be simplicial complexes homeomorphic to $S^n$.
		If $T_1 \ldom T_2$, then $\SV{\Dav{T_1}} \le \SV{\Dav{T_2}}$.
	\end{corollary}
	\begin{proof}
		Let $f:T_2 \to T_1$ be a simplicial map of nonzero degree.
		By \Cref{prop:nonzero_induced}, there is a continuous map $f_M:\Dav{T_2} \to \Dav{T_1}$ with $\abs{\deg(f_M)} \ge 1$.
		Then, we have $\abs{\deg(f_M)} \cdot \SV{\Dav{T_1}} \le \SV{\Dav{T_2}}$, as follows from the classical properties of simplicial volume recalled in \Cref{sec:pre}.
	\end{proof}
	
	To prove \Cref{prop:nonzero_induced} we use a general procedure that, starting from a simplicial map $f$ between simplicial complexes, produces a map $f_M$ between the associated cubical complexes; then, we check that in the situation of \Cref{prop:nonzero_induced} the map $f_M$ has nonzero degree.
	
	Before proving \Cref{lemma:orientations} and \Cref{prop:nonzero_induced}, let us start with a description of the barycentric subdivision of $\Dav{T}$, where $T$ is a finite simplicial complex.
	The map $f_M$ will be defined as a simplicial map between the barycentric subdivisions of $\Dav{T_1}$ and $\Dav{T_2}$.
	
	If $\sigma$ is a simplex of $T$ (thought of as a set of vertices, possibly empty), we denote by
    $\pi_\sigma : [-1,1]^{V(T)} \to [-1,1]^{V(T)}$
    the orthogonal projection on
    $\{(x_v)_{v \in V(T)} \mid x_v = 0 \text{ for every } v \in \sigma\}$.
    In other words, $\pi_\sigma$ sets to $0$ the coordinates corresponding to vertices of $\sigma$.
    
    Take now a vertex $w \in \{-1,+1\}^{V(T)}$ of the ambient cube, along with a strictly increasing sequence $\sigma_0 \subset \dots \subset \sigma_d$ of simplices of $T$. Here we allow $\sigma_0$ to be the empty simplex, and $d$ can be any natural number (but is forced to be at most one more than the dimension of $T$).
    From this data we get a $d$-simplex in $\Dav{T}$ with vertices
    \[ \pi_{\sigma_0}(w), \dots, \pi_{\sigma_d}(w).\]
    In fact these points all lie (and are affinely independent) in the cell of type $\sigma_d$ whose coordinates ``outside $\sigma_d$'' are prescribed by $w$; so we can consider the affine simplex having these points as vertices.
    The simplices obtained in this way will be denoted by $\langle w;\sigma_0, \dots, \sigma_d\rangle$.
    The simplices of this form, obtained by varying the chosen vertex and the chain of simplices of $T$, give a triangulation of $\Dav{T}$ (its barycentric subdivision).
    
    \begin{remark}
    	Two simplices $\langle w;\sigma_0, \dots, \sigma_d\rangle$ and $\langle w';\sigma'_0, \dots, \sigma'_{d'}\rangle$ are equal if and only if $d = d'$, $\sigma_i=\sigma'_i$ for every $i$, and $w_v = w'_v$ for every $v \notin \sigma_0$.
    \end{remark}
	
	\begin{proof}[Proof of \Cref{lemma:orientations}]
		We fix an orientation on $T$ and describe how it induces (canonically) an orientation on $\Dav{T}$.
		We work with simplicial orientations on the barycentric subdivisions of $T$ and $\Dav{T}$.
		
		For the barycentric subdivision of $\Dav{T}$ we use the description introduced in the discussion above; as for $T$, we represent the simplices of its barycentric subdivision as
		$\langle \sigma_0, \dots, \sigma_d \rangle$,
		where $\sigma_0\subset \dots\subset \sigma_d$ is a strictly increasing sequence of simplices of $T$, and $\sigma_0$ is nonempty. This stands for the simplex whose vertices are the barycenters of $\sigma_0,\dots,\sigma_d$.
		
		The simplex $\langle \sigma_0, \dots, \sigma_d \rangle$ has an intrinsic orientation, since its vertices are canonically ordered.
		On the other hand, when $d = n$, it is a top-dimensional simplex of an oriented simplicial sphere, from which it inherits an orientation that may differ from the intrinsic one.
		We set ${\rm or}(\langle \sigma_0, \dots, \sigma_n \rangle)$ equal to $+1$ or $-1$, in case the two orientations agree or disagree, respectively.
		
		We conclude by proving that the following expression gives a simplicial representative of a generator of the top-dimensional homology $H_{n+1}(\Dav{T};\Z)$:
		\[ \sum_{w \in \{-1,+1\}^{V(T)}} \sum_{\langle \sigma_0, \dots,\sigma_n \rangle} \left(\prod_{v \in V(T)} w_v\right) \cdot {\rm or}(\langle \sigma_0,\dots,\sigma_n \rangle) \cdot \langle w;\emptyset,\sigma_0,\dots,\sigma_n \rangle.\]
		Again, we are using the fact that
		$\langle w;\emptyset,\sigma_0,\dots,\sigma_n \rangle$
		has an intrinsic orientation.
		We need to compute the boundary and prove that it vanishes.
		We obtain the sum, over $w$ and $\sigma_0,\dots,\sigma_n$ as above, of the terms
		\begin{align*}
			\left(\prod_{v \in V(T)}w_v\right) &\cdot {\rm or}(\langle \sigma_0,\dots,\sigma_n\rangle) \\
			&\cdot
			\left(
			\langle w;\sigma_0,\dots,\sigma_n \rangle
		 	- \sum_{i=0}^n (-1)^i\langle w;\emptyset,\sigma_0,\dots,\hat{\sigma_i},\dots,\sigma_n \rangle
			\right).
		\end{align*}
		In this expression, the simplex $\langle w;\sigma_0,\dots,\sigma_n \rangle$ cancels with $\langle w';\sigma_0,\dots,\sigma_n \rangle$, where $w'_v = w_v$ for $v \notin\sigma_0$ and $w'_v = -w_v$ for the unique vertex in $\sigma_0$.
		Each simplex of the form $\langle w;\emptyset,\sigma_0,\dots,\hat{\sigma_i},\dots,\sigma_n\rangle$, on the other hand, appears twice in the expression, with opposite ${\rm or}(\langle \sigma_0,\dots,\sigma_n\rangle)$.
	\end{proof}
	
	We now describe how $f:T_1 \to T_2$ induces a continuous map $f_M:\Dav{T_1} \to \Dav{T_2}$.
	Define $f_*:\{-1,+1\}^{V(T_1)} \to \{-1,+1\}^{V(T_2)}$ using the following formula:
	\[ (f_*w)_u \ =\ \prod_{v \in f^{-1}(\{u\})} w_v\]
    for every $w \in \{-1,+1\}^{V(T)}$ and $u \in V(T_2)$.
    Then, we define $f_M$ as a simplicial map between the barycentric subdivisions $\Dav{T_1}$ and $\Dav{T_2}$, sending a vertex of the form $\pi_{\sigma}(w)$ to $\pi_{f(\sigma)}(f_*w)$.
    This definition makes sense because of the following observations:
    \begin{itemize}
    	\item
    	$f(\sigma)$ is a simplex of $T_2$, thus, $\pi_{f(\sigma)}(f_*w)$ is a vertex of the barycentric subdivision of $\Dav{T_2}$;
    	\item
    	If $\pi_\sigma(w) = \pi_\sigma(w')$ are two representations of the same vertex, then $w_v = w'_v$ for every $v \notin \sigma$.
    	Therefore, for every $u \notin f(\sigma)$ we have
    	\[ (f_*w)_u\ =\ \prod_{v \in f^{-1}(\{u\})}w_v\ =\ \prod_{v \in f^{-1}(\{u\})}w_v\ =\ (f_*w')_u, \]
    	so $f_M$ is well defined at least for vertices;
    	\item
    	$f_M$ is simplicial: the simplex $\langle w;\sigma_0,\dots,\sigma_d\rangle$ is sent to the (possibly lower-dimensional) simplex $\langle f_*w;\tau_0,\dots,\tau_k\rangle$, where $\tau_0 \subset \dots \subset \tau_k$ are the simplices of $T_2$ of the form $f(\sigma_i)$ for some $i\in\{0,\dots,d\}$.
    \end{itemize}
    
    \begin{proof}[Proof of \Cref{prop:nonzero_induced}]
    We show that, if $T_1$ and $T_2$ are homeomorphic to $S^n$, then $\deg(f_M)=2^{\abs{V(T_1)}-\abs{V(T_2)}}\cdot\deg(f)$.
    Here, we have fixed arbitrary orientations on $T_1$ and $T_2$, and the induced orientations (as described in the proof of \Cref{lemma:orientations}) on $\Dav{T_1}$ and $\Dav{T_2}$.
    We only treat the case in which $\deg(f) \neq 0$, even though the formula holds also in the zero-degree case.
    This assumption implies that $f:V(T_1)\to V(T_2)$ is surjective, and in particular $\abs{V(T_1)}-\abs{V(T_2)} \ge 0$.
    
    Let $\langle z;\emptyset,\tau_0,\dots,\tau_n\rangle$ be a top-dimensional simplex in $\Dav{T_2}$.
    The set of top-dimensional simplices of $\Dav{T_1}$ whose image under $f_M$ is equal to $\langle z;\emptyset,\tau_0,\dots,\tau_n\rangle$ is in bijection with the set of pairs $(w,\sigma_n)$ where:
    \begin{itemize}
    	\item
    	$w \in \{-1,+1\}^{V(T_1)}$ is such that $f_*w = z$, that is,
    	\[z_u\ =\ \prod_{v\text{ with }f(v)=u} w_v\]
    	for every $u \in V(T_2)$.
    	In particular, it must hold
    	\begin{equation}\label{eq:prod_z}
    		\prod_{v\in V(T_1)} w_v \ =\ \prod_{u\in V(T_2)} z_u;
    	\end{equation}
    	\item
    	$\sigma_n$ is a top-dimensional simplex of $T_1$ with $f(\sigma_n) = \tau_n$.
    \end{itemize}
	Note that any such $\sigma_n$ uniquely determines a simplex $\langle \sigma_0, \dots, \sigma_{n}\rangle$ in the barycentric subdivision of $T_1$ that is sent by $f$ to $\langle \tau_0, \dots, \tau_{n}\rangle$.
    Summing over such $\sigma_n$'s, we obtain
    \begin{equation}\label{eq:sum_or}
    	\deg(f) \cdot {\rm or}(\langle \tau_0,\dots,\tau_n \rangle)\ =\  \sum_{\sigma_n}{\rm or}(\langle \sigma_0,\dots,\sigma_n\rangle).
    \end{equation}
    From \eqref{eq:prod_z}, \eqref{eq:sum_or}, the explicit description of the simplicial orientation cycle in the proof of \Cref{lemma:orientations}, and the fact that $w$ can be chosen in $2^{\abs{V(T_1)}-\abs{V(T_2)}}$ ways (for every fixed $z$), we get the desired formula for $\deg(f_M)$.
    \end{proof}

	\section{Davis' manifolds of dimension 3}\label{sec:3dim}
	In this section we characterize the triangulations of $S^2$ that give rise to three-dimensional manifolds with positive simplicial volume.
	It turns out that the smallest triangulation of $S^2$ with this property is the one described in \Cref{fig:T9}.
	It is a flag triangulation with $9$ vertices, and we call it $T_9$.
	\begin{figure}[ht]
		\centering
		\begin{minipage}{.4\textwidth}
			\centering
			\includegraphics[scale=1]{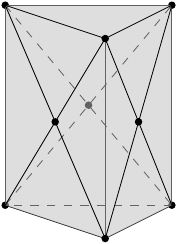}
			\caption{The $T_9$ triangulation of $S^2$.}
			\label{fig:T9}
		\end{minipage}%
		\hspace{.1\textwidth}
		\begin{minipage}{.4\textwidth}
			\centering
			\includegraphics[scale=1]{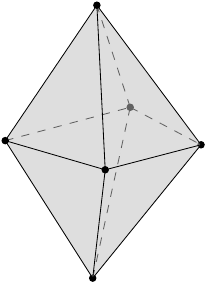}
			\caption{The octahedral triangulation of $S^2$.}
			\label{fig:octa}
		\end{minipage}
	\end{figure}
	The main result of the section is the following.
	\begin{theorem}\label{thm:2dim}
		Let $T$ be a simplicial complex homeomorphic to $S^2$.
		The following conditions are equivalent:
		\begin{enumerate}[label=(\roman*)]
			\item \label{it:2_pos}
			The three-dimensional manifold $\Dav{T}$ has positive simplicial volume;
			\item \label{it:2_t9}
			There is a simplicial map $T\to T_9$ of degree $1$;
			\item \label{it:2_dom}
			$T_9 \ldom T$.
		\end{enumerate}
	\end{theorem}
	The proof consists of two parts:
	\begin{enumerate}[label=(\arabic*)]
		\item
		Showing that $\Dav{T_9}$ has positive simplicial volume, thus obtaining the implication \ref{it:2_dom} $\implies$ \ref{it:2_pos} thanks to \Cref{prop:nonzero_induced}.
		To this aim, we will use the fact that the fundamental group of $\Dav{T_9}$ is relatively hyperbolic;
		\item \label{it:second_part}
		Showing that, if $\Dav{T}$ has positive simplicial volume, then $T$ dominates $T_9$ with a map of degree $1$, thus establishing \ref{it:2_pos} $\implies$ \ref{it:2_t9}.
	\end{enumerate}
	Of course, the implication \ref{it:2_t9} $\implies$ \ref{it:2_dom} is obvious.
	
	We will first prove \ref{it:second_part} in the case where $T$ is flag.
	The strategy will be to employ certain elementary reductions on $T$ that do not affect the positivity $\SV{\Dav{T}}$ and that, starting from any flag triangulation of $S^2$, either lead to the octahedral triangulation (see \Cref{fig:octa}) or to a triangulation satisfying a certain combinatorial property that guarantees condition \ref{it:2_t9}.
	
	Then, to extend the result to a nonflag triangulation $T$, we will cut $T$ along the ``empty triangles'' that prevent $T$ from being flag.
	At the level of $\Dav{T}$, this corresponds to cutting along spheres; \ie, we will reduce the study of $\Dav{T}$ to the study of its prime summands.

	\begin{remark}\label{rmk:rel_hyp}
		In general, a closed aspherical $3$-manifold has positive simplicial volume if and only if its fundamental group is relatively hyperbolic (with respect to a collection of proper subgroups).
		There is a general procedure to decide whether a right-angled Coxeter group is relatively hyperbolic \cite{BHS17}, and thanks to \Cref{prop:relhyp_coxeter} we get an effective way to decide whether $\Dav{T}$ has positive simplicial volume whenever $T$ is a flag triangulation of $S^2$.
		The point of \Cref{thm:2dim} is not so much obtaining such an effective procedure, but to show that the triangulations giving positive simplicial volume all dominate (in the sense of \Cref{def:dominance}) a single triangulation that can then be studied explicitly.
		This motivates us to try to employ the same strategy also in higher dimensions.
		Note that, in dimension $\ge 4$, the simplicial volume can be positive even when the fundamental group is not relatively hyperbolic.
	\end{remark}
	
	\begin{proposition}
		The simplicial volume of $\Dav{T_9}$ is positive.
	\end{proposition}
	\begin{proof}
		Once we establish that the fundamental group of $\Dav{T_9}$ is relatively hyperbolic, the conclusion follows from \Cref{prop:relhyp_sv}.
		Let $\Gamma$ be the right-angled Coxeter group defined by the $1$-skeleton of $T_9$.
		By \Cref{prop:relhyp_coxeter} it is enough to show that $\Gamma$ is relatively hyperbolic.
		This is easily checked using the algorithmic procedure described in \cite{BHS17}, which can be used for any right-angled Coxeter group.
		The thick subgroups (this is the terminology used in \cite{BHS17}) of $\Gamma$, with respect to which $\Gamma$ is relatively hyperbolic, are those corresponding to the three subcomplexes given by the lateral squares in \Cref{fig:T9}.
	\end{proof}
	
	To prove the remaining implication of \Cref{thm:2dim}, we start by introducing a way of simplifying a flag triangulation of $S^2$ without altering the positivity of the simplicial volume of the corresponding $3$-manifold.
	\begin{definition}
		Let $T$ be a flag triangulation of $S^2$, and suppose that it contains a subcomplex isomorphic to the one in \Cref{fig:s2bigp}, in such a way that the vertices labelled with $A$ and $B$ are not adjacent in $T$.
		An elementary reduction consists in replacing this subcomplex with the one in \Cref{fig:s2smallp}.
		In other words, it is the collapse of an edge whose vertices have valence $4$.
	\end{definition}
	\begin{figure}[ht]
		\centering
		\begin{minipage}{.4\textwidth}
			\centering
			\includegraphics[scale=1]{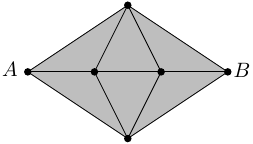}
			\caption{The subcomplex to be replaced.}
			\label{fig:s2bigp}
		\end{minipage}%
		\hspace{.1\textwidth}
		\begin{minipage}{.4\textwidth}
			\centering
			\includegraphics[scale=1]{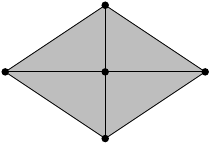}
			\caption{The subcomplex that replaces the previous one.}
			\label{fig:s2smallp}
		\end{minipage}
	\end{figure}
	The only case in which the vertices labelled $A$ and $B$ can be adjacent (so, we do \emph{not} allow the collapse) is the one in which $T$ is isomorphic to the octahedral triangulation.
	It is readily seen that the result $T'$ of an elementary reduction is another flag triangulation of $S^2$, and that there is a simplicial map $T\to T'$ of degree $1$.
	
	\begin{lemma}\label{lemma:elementary_red}
		If $T'$ is obtained from a flag triangulation $T$ by an elementary reduction, then $\SV{\Dav{T}} = 2 \SV{\Dav{T'}}$.
	\end{lemma}
	\begin{proof}
		Let $R$ be the subcomplex of $T$ in \Cref{fig:s2bigp} that has to be replaced in order to obtain $T'$, and let $S$ be the subcomplex of $T$ obtained by removing the interior of $R$.
		Let $k$ be the number of vertices of $S$, among which there are also the $4$ vertices of the square ``cutting locus'' of the reduction.
		The cutting locus, which is a square in $T$, corresponds to $2^{k-2}$ disjoint embedded tori in $\Dav{T}$, that cut $\Dav{T}$ into:
		\begin{itemize}
			\item $4$ embedded copies of $\Dav{S}$, the latter being a compact $3$-manifold with boundary;
			\item $2^{k-4}$ embedded copies of $\Dav{R}$, each of which is homeomorphic to the product of $S^1$ with an oriented connected surface with genus $1$ and $4$ boundary components.
		\end{itemize}
		Each of the $2^{k-2}$ tori is $\pi_1$-injective in the two pieces of which it is a boundary component.
		Let $R'$ be the subcomplex of $T'$ that replaces $R$.
		As above, we can cut $\Dav{T'}$ along tori and obtain:
		\begin{itemize}
			\item $2$ embedded copies of $\Dav{S}$;
			\item $2^{k-4}$ embedded copies of $\Dav{R'}$, each of which is homeomorphic to the annulus $S^1 \times [-1,1]$.
		\end{itemize}
		The conclusion follows from Gromov's additivity theorem for simplicial volume (see, \eg, \cite[Theorem 7.6]{Frigerio2017}), since the only pieces contributing to the simplicial volume are those homeomorphic to $\Dav{S}$ (the others have an $S^1$ factor).
	\end{proof}
	
	\begin{proposition}\label{prop:reduction_t9}
		Let $T$ be a flag triangulation of $S^2$ such that every edge has at least an endpoint whose valence is at least $5$ (equivalently, $T$ is a flag triangulation on which no elementary reduction can be performed, and $T$ is not the octahedral triangulation).
		Then there is a simplicial map of degree $1$ from $T$ to $T_9$. 
	\end{proposition}
	\begin{proof}
		Let us suppose by contradiction that there is a triangulation $T$ satisfying the assumption in the statement, but that does not admit simplicial maps to $T_9$ of degree $1$.
		We take $T$ to be a counterexample with the smallest number of vertices.
		
		\paragraph{Case 1:}\label{case1} $T$ does not contain squares.
		
		Take an arbitrary $1$-simplex and collapse it, obtaining a new triangulation $T'$ to which $T$ reduces via a degree-$1$ simplicial map (see \Cref{fig:collapse_edge}).
		The new triangulation $T'$ is flag, because the collapsed edge is not contained in a square.
		The only vertices in $T'$ that could have valence $4$ are those labelled by $X$ and $Y$ in \Cref{fig:collapse_edge}, because all vertices in $T$ have valence at least $5$.
		Moreover, $X$ and $Y$ are not adjacent in $T'$: if they were, then $T'$ would not be flag.
		This contradicts the minimality of $T$.
		\begin{figure}[ht]
			\centering
			\includegraphics[scale=1]{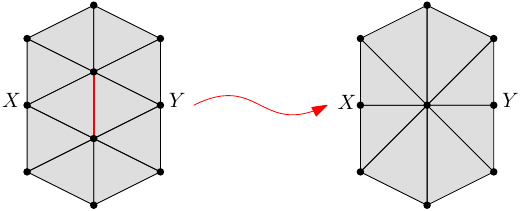}
			\caption{The collapse of an edge.}
			\label{fig:collapse_edge}
		\end{figure}
		
		\paragraph{Case 2:} $T$ contains a square which is not a vertex link.
		
		Fix one such square, that we call equator.
		It divides the sphere in a southern and a northern hemisphere (see \Cref{fig:equator_s2}).
		Each of the two hemispheres has at least $2$ internal vertices, otherwise the equator would be a vertex link.
		If a vertex $X$ of the equator has only one adjacent vertex $Y$ in the northern hemisphere, we can consider another equator, replacing $X$ with $Y$ (see \Cref{fig:equator_move}).
		This new equator is not a vertex link, because on the southern hemisphere there are at least $3$ vertices, and if on the northern side there were only one vertex $Z$, then $Y$ and $Z$ would be two adjacent vertices of valence $4$.
		After performing this change of equator a finite number of times, we can assume that each vertex in the equator is adjacent to at least $2$ vertices in the northern hemisphere.
		Now, consider the triangulation $T'$ obtained by replacing the southern hemisphere with a complex as in \Cref{fig:s2smallp}.
		There is a natural simplicial map of degree $1$ from $T$ to $T'$, so there cannot be a simplicial map from $T'$ to $T_9$ of degree $1$.
		Since every vertex of the equator has valence at least $5$ in $T$, there is no pair of adjacent vertices of valence $4$ in $T'$ (since this was true in $T$).
		Moreover, $T'$ is flag: any set of pairwise-adjacent vertices must be contained in one of the two hemispheres, and both hemispheres are flag.
		This contradicts the minimality of $T$.
		\begin{figure}[ht]
			\centering
			\begin{minipage}{.4\textwidth}
				\centering
				\includegraphics[scale=1]{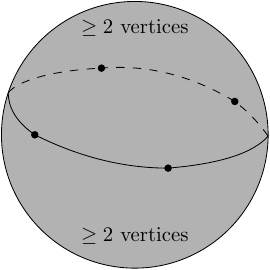}
				\caption{The equator and the two hemispheres.}
				\label{fig:equator_s2}
			\end{minipage}%
			\hspace{.1\textwidth}
			\begin{minipage}{.4\textwidth}
				\centering
				\includegraphics[scale=1]{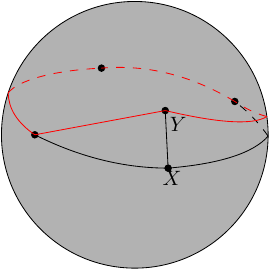}
				\caption{The new equator.}
				\label{fig:equator_move}
			\end{minipage}
		\end{figure}
		
		\paragraph{Case 3:} $T$ contains some squares, but all of them are vertex links.
		
		Fix arbitrarily one of these squares, and call $V$ the vertex of which it is the link.
		If we could collapse one of the edges emanating from $V$ and obtain another triangulation satisfying the assumptions in the statement, we would contradict the minimality of $T$.
		The result of such collapses is always flag (see \Cref{fig:flag_assured}), so these collapses must produce adjacent vertices of valence $4$.
		\begin{figure}[ht]
			\centering
			\includegraphics[scale=1]{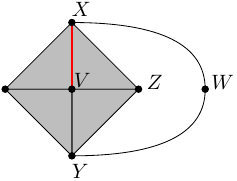}
			\caption{Suppose that collapsing the edge $VX$ we get a nonflag triangulation.
			This implies that $X$ and $Y$ are both adjacent to some other vertex $W$.
			The square $VXWY$ must be a vertex link, so $W$ is adjacent to $Z$ (or to the opposite vertex, leading to an analogous situation).
			But now $Z$ and $V$ are adjacent and of valence $4$, which is a contradiction.}
			\label{fig:flag_assured}
		\end{figure}
		In order for this to happen, one of the four vertices of the square (let us call it $X$) must have valence $5$, and at least one of its two neighbours which are not adjacent to $V$ must have valence $4$ (see \Cref{fig:redt9_1}).
		\begin{figure}[ht]
			\centering
			\includegraphics[scale=1]{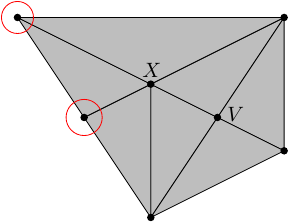}
			\caption{One of the two circled vertices must have valence $4$.
			Because of flagness, and the ``squares are vertex links'' and ``valence-$4$ vertices are not adjacent'' conditions, the vertices drawn up to this point are all distinct, and two of them are adjacent exactly when they are so in this picture.}
			\label{fig:redt9_1}
		\end{figure}
		The current situation is summarized in \Cref{fig:redt9_2}, which starts to look very similar to the picture of $T_9$ in \Cref{fig:T9}, of which only the centre of the back square is missing.
		\begin{figure}[ht]
			\centering
			\includegraphics[scale=1]{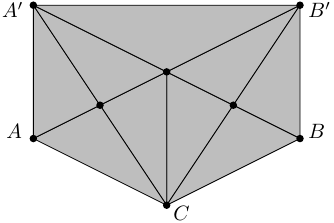}
			\caption{The vertices drawn here are all distinct, but there might be an additional edge between $A$ and $B$.
			Not between $A$ and $B'$, as this would create a square link around $A'$, which cannot have valence $4$ as it is adjacent to another vertex of valence $4$.
			For a similar reason $A'$ and $B$ are not adjacent.}
			\label{fig:redt9_2}
		\end{figure}
		We consider the simplicial map $f:T \to T_9$ defined as follows:
		\begin{itemize}
			\item The vertices appearing in \Cref{fig:redt9_2} are sent to the corresponding vertices of $T_9$ as suggested by superimposing \Cref{fig:redt9_2} and \Cref{fig:T9};
			\item Among the other vertices, those not adjacent to $C$ are sent to the centre of the back square in \Cref{fig:T9};
			\item The remaining vertices are sent to $f(A)$ (here, $f(B)$ would work equally well).
		\end{itemize}
		This actually gives a simplicial map (the image of adjacent vertices are either equal or adjacent), the only nontrivial check being that vertices of the last type are not adjacent to $B'$.
		But this is true, for the same argument used in the caption of \Cref{fig:redt9_2} to show that $A$ and $B'$ are not adjacent.
		The map $f$ is clearly of degree $1$, and this concludes the proof.
	\end{proof}
	
	\begin{proof}[Proof of the implication \ref{it:2_pos} $\implies$ \ref{it:2_t9} of \Cref{thm:2dim}]
		Suppose first that $T$ is flag.
		Up to performing a sequence of elementary reductions, which preserve the positivity of simplicial volume by \Cref{lemma:elementary_red}, we may suppose that either $T$ is the octahedral triangulation or it satisfies the assumption of \Cref{prop:reduction_t9}.
		In the first case $\SV{\Dav{T}} = 0$, because the cubical complex corresponding to the octahedral triangulation is homeomorphic to the three-dimensional torus, whose simplicial volume is $0$.
		In the second case we conclude by \Cref{prop:reduction_t9}.
		
		We now remove the assumption of flagness, and proceed by induction on the number of vertices of $T$.
		For the base case we consider the triangulation of $S^2$ with $4$ vertices, as the boundary of a tetrahedron.
		The cubical complex associated to it is homeomorphic to $S^3$, and its simplicial volume vanishes.
		
		For the induction step, if a given triangulation $T$ is flag, then we conclude as in the first part of the proof.
		Otherwise, it must contain an ``empty triangle'', \ie, a subcomplex isomorphic to the boundary of a $2$-simplex, but which is not the boundary of a $2$-simplex of $T$.
		Fix such an empty triangle $E$ (standing for ``equator''), and consider the two hemispheres $H_1,H_2$ it cuts $T$ into.
		Replacing $H_1$ with the interior of a $2$-simplex, we get a triangulation $T_1$ having fewer vertices than $T$.
		Analogously, we get a triangulation $T_2$ by replacing $H_2$ (while keeping $H_1$ intact).
		Now, $\Dav{T}$ is obtained by taking copies of $\Dav{T_1}$ and $\Dav{T_2}$, removing from them some $3$-balls, and then gluing the resulting spherical boundary components in an appropriate way.
		By Gromov's additivity theorem, if $\SV{\Dav{T}} > 0$ then at least one of $\SV{\Dav{T_1}}$ and $\SV{\Dav{T_2}}$ must be positive.
		By the induction hypothesis, either $T_1$ or $T_2$ must dominate $T_9$ with a degree-one map.
		Since $T$ dominates both $T_1$ and $T_2$ with a degree-one map, the conclusion follows.
	\end{proof}

	\section{Davis' manifolds of dimension 4}\label{sec:4dim}
	With the four-dimensional manifolds arising from the Davis' construction, we can finally start to test \Cref{conj:sv_chi}.
	
	From \Cref{lemma:chi_formula} we know that, if $T$ is a flag triangulation of $S^3$, its Euler characteristic is nonzero if and only if the quantity
	\[ \gamma_2(T) = f_3(T) - 2f_2(T) + 4f_1(T) - 8f_0(T) + 16 \]
	is nonzero, where $f_i(T)$ denotes the number of $i$-dimensional simplices in $T$.
	Using the relations $2f_3(T) = f_2(T)$ and $f_3(T) - f_2(T) + f_1(T) - f_0(T) = 0$, we can rewrite
	\[ \gamma_2(T) = 16 - 5f_0(T) + f_1(T). \]
	The notation we are using is standard in the literature; the number $\gamma_2$ is in fact a component of a vector of numbers, the $\gamma$-vector, associated to any triangulation of a sphere (of any dimension), or more generally to any \emph{Eulerian complex}.
	We refer the reader to \cite{Atha2016} for a survey on the $\gamma$-vector and its relevance in several geometric and combinatorial contexts.
	
	It is known that $\gamma_2(T) \ge 0$ for every flag triangulation $T$ of $S^3$; in this sense, the ones for which $\Dav{T}$ has vanishing Euler characteristic can be considered as the extremal ones.
	This fact was proved by Davis and Okun \cite{DO2001}, and their proof requires the machinery of $\ell^2$-homology.
	The use of $\ell^2$-homology shouldn't surprise the reader, because $\ell^2$-homology, like ordinary homology, can be used to compute Euler characteristics, and the condition $\gamma_2(T) \ge 0$ is equivalent to $\chi(M(T)) \ge 0$.
	Nevertheless, the fact that they have to resort to such machinery, despite the statement being quite elementary, indicates that this result is deeper than it might appear at a first glance.
	The theorem of Davis and Okun is a particular case of the Euler characteristic conjecture asserting that, if $M$ is a closed aspherical manifold of dimension $2k$, then its Euler characteristic satisfies $(-1)^k\chi(M) \ge 0$.
	It is also a particular case of a major open question in geometric combinatorics, asking whether \emph{the whole} $\gamma$-vector of flag spheres is nonnegative.
	
	\begin{definition}\label{def:minimal}
		Let $T$ be a flag triangulation of $S^3$ with $\gamma_2(T) > 0$.
		We say that $T$ is \emph{\minit{}} if it does not dominate any other triangulation with the same properties (being flag and having positive $\gamma_2$).
	\end{definition}
	
	We attempt the following two-steps program to check \Cref{conj:sv_chi} for four\nobreakdash{-}\hspace{0pt}dimensional manifolds obtained via the Davis' construction:
	\begin{itemize}
		\item Find the \minit{} triangulations of $S^3$;
		\item Check if the simplicial volumes of the manifolds associated to the \minit{} triangulations are positive.
	\end{itemize}
	
	Let us start with some simple observations.
	\begin{remark}
		If $T$ is a flag triangulation of $S^3$ and $v$ is a vertex of $T$, the link of $v$ is a full subcomplex of $T$, and, therefore, it is flag.
		Hence, vertex links (and, in general, full subcomplexes homeomorphic to $S^2$) are flag triangulations of $S^2$.
	\end{remark}
	\begin{remark}
		By induction, it is readily seen that for every $n \ge 0$ there is exactly one (up to isomorphism of simplicial complexes) flag triangulation of $S^n$ with $2n+2$ vertices: the $(n+1)$-fold join of the (unique) triangulation of $S^0$, which is called the octahedral $n$-sphere.
		It is the flag triangulation of $S^n$ with the lowest number of vertices.
		For $n = 1$ it is a square, for $n = 2$ it is an octahedron (\Cref{fig:octa}).
		For $n = 3$ it has $8$ vertices, and it is the only flag triangulation of $S^3$ whose vertex links are all octahedra.
	\end{remark}
	\begin{remark}
		Recall that the suspension of a simplicial complex $S$ is defined as the join $S \star S^0$.
		If $T$ is the suspension of a (flag) triangulation of $S$ of $S^2$, then $T$ is a (flag) triangulation of $S^3$ and an easy computation shows that $\gamma_2(T) = 0$.
		Notice also that $\Dav{T} = \Dav{S \star S^0} = \Dav{S} \times \Dav{S^0} = \Dav{S} \times S^1$ has vanishing simplicial volume.
	\end{remark}
	\begin{remark}\label{rem:9vert}
		If $T$ is a flag triangulation of $S^3$ with $9$ vertices, one of its vertex links must have $7$ vertices (not being isomorphic to an octahedron), so $T$ is a suspension and $\gamma_2(T) = 0$.
	\end{remark}
	
	\subsection{The join of two pentagons}
	Hereafter, we denote by $T_{10}$ the join of two pentagons.
	It is a flag triangulation of $S^3$ with $10$ vertices and $35$ edges.
	In particular, $\gamma_2(T_{10}) = 16 - 5\cdot 10 + 35 = 1 > 0$.
	By the remarks above, it has the lowest possible number of vertices for a flag triangulation of $S^3$ with nonvanishing $\gamma_2$, so it is \minit{}.
	
	Recall that the join operation translates into a direct product at the level of cubical complexes, and that a pentagon gives a surface admitting hyperbolic metrics, which has positive simplicial volume.
	This implies that $\Dav{T_{10}}$ has positive simplicial volume, being the product of two surfaces with positive simplicial volume.
	
	The following lemma gives a simple sufficient condition for a triangulation of $S^3$ to dominate $T_{10}$.
	\begin{lemma}\label{lemma:edgelink5}
		Let $T$ be a flag triangulation of $S^3$ and suppose that it has an edge that is not contained in any square and whose link (which is a triangulation of $S^1$) is not a square.
		Then $T \gdom T_{10}$.
	\end{lemma}
	\begin{proof}
		A simplicial map (of degree $\pm 1$) can be defined as described in \Cref{fig:edgelink5}, where the vertices labelled with $A$ and $B$ are the endpoints of an edge satisfying the assumptions in the statement.
		\begin{figure}
			\centering
			\includegraphics[scale=1]{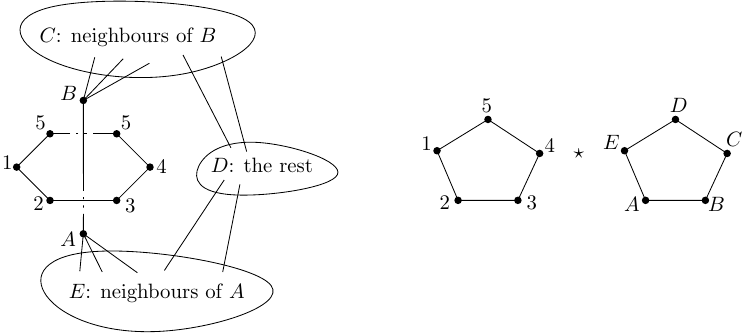}
			\caption{Schematic description of a simplicial map from $T$ (on the left) to the join of two pentagons (on the right).
			Each vertex is sent to the one labelled with the same symbol.}
			\label{fig:edgelink5}
		\end{figure}
	\end{proof}
	
	\begin{example}
		Let $T$ be the barycentric subdivision of a simplicial complex homeomorphic to $S^3$.
		Then $T \gdom T_{10}$, because any edge joining the barycentre of a $3$-simplex to the barycentre of a $2$-simplex satisfies the assumption in \Cref{lemma:edgelink5}.
	\end{example}
	
	\begin{example}
		Let $T$ be a flag triangulation of $S^3$ such that $\Dav{T}$ has hyperbolic fundamental group (this implies positivity of simplicial volume).
		Then $T \gdom T_{10}$, again by \Cref{lemma:edgelink5}, because $T$ does not contain squares: a square would induce a subgroup of $\pi_1(\Dav{T})$ isomorphic to $\Z^2$, but hyperbolic groups cannot have subgroups isomorphic to $\Z^2$.
	\end{example}
	
	\begin{question}
		Let $T$ be a flag triangulation of $S^3$ and suppose that $\pi_1(\Dav{T})$ is relatively hyperbolic.
		Is it true that $T \gdom T_{10}$?
		This is not unreasonable since relative hyperbolicity imposes restrictions on the collection of squares (see \cite{BHS17}).
	\end{question}
	
	Let us consider a special class of triangulations.
	\begin{definition}\label{def:incremental}
		Let $T$ be a flag triangulation of $S^3$; we say that $T$ is \emph{incremental} if it contains a sequence of nested full subcomplexes $S_1 \subset \dots \subset S_k$ such that:
		\begin{itemize}
			\item $S_1$ is the $1$-neighbourhood (the closed star) of a vertex, and $S_k = T$;
			\item For every $i < k$, the subcomplex $S_{i+1}$ has exactly one vertex more than $S_i$;
			\item For every $i < k$, the subcomplex $S_i$ is homeomorphic to the $3$-ball and its boundary is a full subcomplex (of $S_i$ or of $T$; there is no difference since $S_i$ is full in $T$).
		\end{itemize}
	\end{definition}
		Equivalently, incremental triangulations are those obtained with the following procedure:
		\begin{enumerate}
			\item Start with the cone over a flag triangulation of the $2$-sphere;
			\item Perform the following operation a finite number of times: if $S$ is the current simplicial complex (which is a flag triangulation of the $3$-ball, whose boundary is a full subcomplex), choose a full subcomplex of $\partial S$ homeomorphic to the $2$-disc and whose boundary is full, and attach a cone (introducing a new vertex as the tip of the cone) over this subcomplex;
			\item Finally, attach a cone over the boundary of the current simplicial complex.
		\end{enumerate}
		For example, the suspension of a flag triangulation of $S^2$ is incremental (with $k = 2$); also, it is not hard to see that if $T$ has a vertex which is adjacent to all but $2$ of the other vertices, then it is incremental (with $k = 3$).
	\begin{theorem}\label{thm:incremental}
		If $T$ is incremental and $\gamma_2(T) > 0$, then $T \gdom T_{10}$. 
	\end{theorem}
	\begin{proof}		
		For $S$ a triangulation of the $3$-ball, consider the following quantity:
		\[ \theta(S) = 11 - 5f_0(S) + f_1(S) + f_0(\partial S).\]
		Then the changes of $\theta$ during the procedure described in the construction of an incremental triangulation are as follows:
		\begin{enumerate}
			\item At the beginning, when $S$ is the cone over a $2$-sphere, we have $\theta(S) = 0$;
			\item When we attach a cone to $S$ obtaining a new $S'$, we obtain $\theta(S') = \theta(S) + m - 4$, where $m$ is the number of vertices in the boundary of the disc over which the cone is attached;
			\item When we finally obtain $T$ by attaching a cone over the whole boundary of $S$, we get $\gamma_2(T) = \theta(S)$.
		\end{enumerate}
		Therefore, $\gamma_2(T)$ is positive if and only if in one of the steps the boundary of the disc has at least $5$ vertices, and if this is the case then $T$ dominates $T_{10}$ as follows (the picture is not dissimilar to \Cref{fig:edgelink5}):
		\begin{itemize}
			\item The boundary of the disc selected in the step when $\theta$ increases is mapped to the first pentagon of the join;
			\item The other vertices are sent to the vertices of the second pentagon, in cyclic order, partitioned as: the tip of the newly attached cone; the interior of the disc; the interior of the $3$-ball; the boundary of the current $S$, excluding the disc; vertices that will be added in subsequent steps.
		\end{itemize}
		Thus, if $\gamma_2(T) > 0$, then $T$ dominates $T_{10}$.
	\end{proof}
		
	\subsection{Triangulations with few vertices}
	We have already seen that flag triangulations with up to $9$ vertices are suspensions.
	We now consider flag triangulations $T$ with up to $11$ vertices.
	
	If $T$ has $10$ vertices, then it must have a vertex $X$ whose valence is at least $7$, because vertex links cannot be all octahedra.
	This leaves at most $2$ vertices outside the closed star of $X$; we have already observed that this condition implies that $T$ is incremental (\Cref{def:incremental}).
	Hence, either $\gamma_2(T) = 0$ or $T \gdom T_{10}$ by \Cref{thm:incremental}, but in the latter case it must be isomorphic to $T_{10}$.
	We conclude that $T_{10}$ is the only flag triangulation (up to isomorphisms) with $10$ vertices with positive $\gamma_2$.
	
	If $T$ has $11$ vertices, and one of them has valence at least $8$, we conclude as above that either $\gamma_2(T) = 0$ or $T \gdom T_{10}$.
	If all of its vertices have valence $6$ or $7$, it actually follows from the inequality $\gamma_2(T) \ge 0$ that exactly one of the vertices has valence $6$ and the others have valence $7$; this gives a total of $39$ edges and $\gamma_2(T) = 0$.
	We conclude the following:
	\begin{theorem}\label{thm:less12}
		$T_{10}$ is the only \minit{} triangulation with up to $11$ vertices.
	\end{theorem}
	
	\subsection{A \minit{} triangulation with 12 vertices}
	In this subsection we introduce another flag triangulation.
	It has $12$ vertices, which we label with numbers from $0$ to $11$.
	We denote this triangulation by $T_{12}$; its $1$-skeleton is schematically described in \Cref{fig:T12}, and has $45$ edges.
	Recall that the whole simplicial complex can be reconstructed from its $1$-skeleton, because of flagness.
	\begin{figure}
		\centering
		\includegraphics[scale=1]{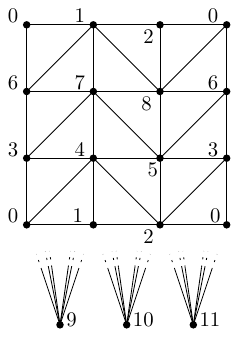}
		\caption{The $1$-skeleton of $T_{12}$. Some vertices are drawn multiple times, as suggested by the repeated labels, for convenience. Each of the three vertices labelled $9$, $10$ and $11$ is adjacent to the six vertices in the column above it.}
		\label{fig:T12}
	\end{figure}
	
	It is not so easy to get a grasp on the complex $T_{12}$.
	It has already been considered by Venturello \cite{Venturello22} as an example of a flag triangulation of $S^3$ which is not a suspension, whose full subcomplexes homeomorphic to the $2$-sphere are all vertex links, and whose edges are all contained in squares.
	In particular he proves that $T_{12}$ is a triangulation of $S^3$, a fact that isn't clear at first glance.\footnote{Venturello uses a different construction to obtain a flag triangulation of $S^3$ that he calls $\Delta_{12,33}$. It is easy to compare his construction with ours and to check that $\Delta_{12,33} \cong T_{12}$.}
	
	Notice that, from the number of vertices and edges, we get $\gamma_2(T_{12}) = 16 - 5\cdot 12 + 45 = 1 > 0$.
	We prove in \Cref{prop:T12_minimal} below that $T_{12}$ is \minit{}.

	\paragraph{Symmetries of $\bm{T_{12}}$.}
	The triangulation $T_{12}$ has some nontrivial automorphisms, induced by the following permutations of $V(T_{12})$:
	\begin{itemize}
		\item $(0 \to 2 \to 4 \to 3 \to 5 \to 7 \to 6 \to 8 \to 1 \to 0), (9 \to 10 \to 11 \to 9)$;
		\item $(3 \leftrightarrow 6), (1 \leftrightarrow 2), (7 \leftrightarrow 5), (4 \leftrightarrow 8), (9 \leftrightarrow 11)$;
		this can be pictured by performing a half-turn rotation in \Cref{fig:T12}.
	\end{itemize}
	These automorphisms identify two orbits of vertices: the ones numbered $0$ through $8$ and the ones from $9$ to $11$.
	The edges, instead, are divided into four orbits, represented for instance by $\{10,4\}$, $\{7,4\}$, $\{4,5\}$ and $\{7,5\}$.

	\begin{proposition}\label{prop:T12_minimal}
		The triangulation $T_{12}$ is \minit{}.
	\end{proposition}
	In general, if $T' \ldom T$, then either $T'$ has less vertices than $T$, or $T'$ is isomorphic to $T$.
	Since $T_{10}$ is the only \minit{} triangulation with less than $12$ vertices (\Cref{thm:less12}), we only have to prove that $T_{10} \not\ldom T_{12}$.
	This can be verified by enumerating all the possible maps from $V(T_{12})$ to $V(T_{10})$ (which are finitely many) and checking that none of them is a simplicial map of nonzero degree.
	The amount of such maps is quite big, but with a smart enough implementation one can perform these checks in a matter of minutes on a computer.
	In any case, we provide here a compact proof which is understandable by humans.
	\begin{proof}
		
		Suppose that there is a simplicial map $f:T_{12} \to T_{10}$ with $\deg(f) \neq 0$; we will derive a contradiction from this assumption.
		
		Since $\deg(f) \neq 0$, the map $f$ is surjective.
		There must be either a vertex of $T_{10}$ whose preimage has $3$ vertices (while the preimages of other vertices contain exactly one vertex each), or two distinct vertices of $T_{10}$ whose preimages have two vertices each.
		The four possible situations are summarized in \Cref{fig:T12_T10}.
		
		\begin{figure}
			\centering
			\includegraphics[scale=1]{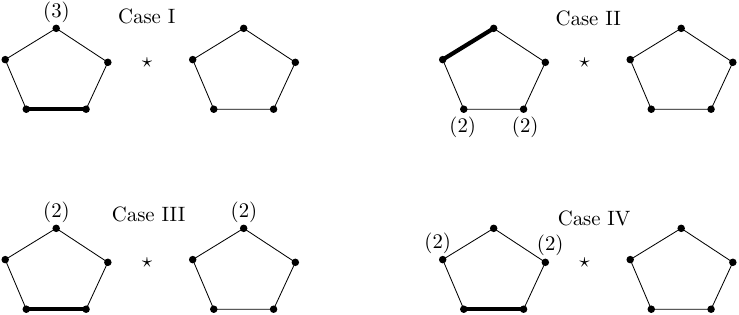}
			\caption{The four essentially distinct ways in which the cardinalities of preimages can be distributed among the vertices of $T_{10}$.}
			\label{fig:T12_T10}
		\end{figure}
		
		Cases I, II and III can be treated as follows.
		In each of these cases, there is an edge $e$ (marked with a heavier line in \Cref{fig:T12_T10}) such that:
		\begin{itemize}
			\item The edge $e$ is contained in one of the two pentagons;
			\item Its endpoints have preimages of cardinality 1;
			\item The vertex $v$ lying on the same pentagon as $e$ and opposite to it has preimage of cardinality at least 2.
		\end{itemize}
		The preimage of $e$ must be an edge in $T_{12}$.
		Notice that the link of $e$ in $T_{10}$ is a pentagon; since $\deg(f) \neq 0$, the link of $f^{-1}(e)$ must have at least $5$ vertices.
		Up to composing with an automorphism of $T_{12}$, we can assume that $f^{-1}(e) = \{4,7\}$, or $f^{-1}(e) = \{5,7\}$ because the edges in the other $\rm{Aut}(T_{12})$-orbits have links with just $4$ vertices.
		If $f^{-1}(e) = \{4,7\}$, a contradiction arises because the only vertex in $T_{12}$ which is adjacent to neither $4$ nor $7$ is the one labelled $11$, but $f$ should send more than one vertex to $v$.
		If $f^{-1}(e) = \{5,7\}$, we conclude similarly since the only vertex adjacent to neither $5$ nor $7$ is $0$.
		
		We now consider Case IV.
		Again, we can assume that the edge $e$ marked in \Cref{fig:T12_T10} has preimage $f^{-1}(e) = \{4,7\}$ or $f^{-1}(e) = \{5,7\}$.
		Suppose that $f^{-1}(e) = \{4,7\}$.
		The link of $\{4,7\}$ in $T_{12}$ contains the vertices $1, 10, 5, 3, 9$; they must be mapped by $f$ to the five vertices of the second pentagon, which is the link of $e$.
		The other vertices of $T_{12}$ must be mapped to the first pentagon.
		Consider now vertex $6$: it is adjacent to $7$, so $f(6)$ is one of the two vertices with double preimage (the one adjacent to $f(7)$).
		On the other hand, vertex $0$ is adjacent to $4$, so $f(0)$ is the \emph{other} vertex with double preimage.
		A contradiction arises because $0$ and $6$ are adjacent in $T_{12}$, but $f(0)$ and $f(6)$ are not adjacent in $T_{10}$.
		If $f^{-1}(e) = \{5,7\}$ we conclude similarly, considering the vertices $1$ and $2$ in place of $6$ and $0$.
	\end{proof}
	
	Recall that our strategy for testing \Cref{conj:sv_chi} has two steps: finding the basic triangulations and checking that they give manifolds of positive simplicial volume.
	
	\begin{question}\label{q:T12_positive}
		Does $\Dav{T_{12}}$ have positive simplicial volume?
	\end{question}
	We aren't able to answer this (and the following) question in this paper.
	Using the algorithm in \cite{BHS17} along with \Cref{prop:relhyp_coxeter}, it is routine to check that the fundamental group of $\Dav{T_{12}}$ is not relatively hyperbolic.
	We would need to find other reasons for the simplicial volume to be positive, in order to answer (positively) to \Cref{q:T12_positive}.
	\begin{question}\label{q:other_min}
		Are $T_{10}$ and $T_{12}$ the only \minit{} triangulations of $S^3$?
	\end{question}

	\subsection{The quest for other \minit{} triangulations}\label{ssec:quest}
	We start by pointing out a condition that any \minit{} triangulation distinct from $T_{10}$ must satisfy.
	Recall that a square in $T$ is a full subcomplex of $T$ homeomorphic to $S^1$ having four vertices.
	
	\begin{proposition}\label{prop:many_squares}
		Let $T$ be \minit{}, not isomorphic to $T_{10}$.
		Then every edge of $T$ is contained in a square.
	\end{proposition}
	\begin{proof}
		Suppose that there is an edge not contained in a square.
		Then, by collapsing it, we obtain another flag triangulation $T'$.
		It is easy to see that $\gamma_2(T') = \gamma_2(T) + 4 - k$, where $k$ is the number of vertices in the link of the collapsed edge.
		Since $T$ is \minit{}, it must hold $\gamma_2(T') = 0$.
		Since $\gamma_2(T) \neq 0$, the number $k$ cannot be equal to $4$.
		But then $T \gdom T_{10}$ by \Cref{lemma:edgelink5}.
	\end{proof}
	
	There are triangulations in which every edge is contained in a square, but which nonetheless dominate $T_{10}$.
	In fact, a computer search suggests that most of them dominate $T_{10}$ (see Subsection \ref{ssec:results}).
	
	\begin{proposition}\label{prop:infinite_manysquare}
		There are infinitely many pairwise-nonisomorphic flag triangulations of $S^3$ with positive $\gamma_2$ in which every edge is contained in a square.
	\end{proposition}
	\begin{proof}
		We describe how, given such a triangulation $T$, we can construct another triangulation satisfying the same properties and having more vertices.
		Then, we can build infinitely many by starting from $T = T_{12}$ and applying the construction arbitrarily many times.
		
		Pick a vertex $v$ of $T$, and let $S \subset T$ be the link of $v$, which is a triangulated $2$-sphere.
		Let $T_1$ and $T_2$ be two copies of $T$ in which $v$ (along with its open star) has been removed: they are triangulations of a $3$-dimensional disk, with boundary isomorphic to $S$.
		Define $T'$ by gluing $T_1$ and $T_2$ along the canonical identification of their boundary.
		\begin{itemize}
			\item \emph{$T'$ is flag}.
			In fact, any set of pairwise-adjacent vertices in $T'$ is contained either in $T_1$ or in $T_2$, and both $T_1$ and $T_2$ are flag complexes.
			\item \emph{Every edge of $T'$ is contained in a square.}
			Let $\{x,y\}$ be an edge in $T'$. We can assume without loss of generality that $x$ and $y$ belong to $T_1$.
			We denote by $x$ and $y$ also the corresponding vertices of $T$.
			If $x$ and $y$ are part of a square in $T_1$, we are done.
			Suppose this is not the case; they must be part of a square in $T$, and this square must contain the vertex $v$.
			Exactly one of $x$ and $y$ must lie in $S$ --- not both, otherwise in $T$ they would form a triangle with $v$, which is impossible if $x$, $y$ and $v$ belong to a square --- let us say that $x$ does. The square contains a fourth vertex $w$ lying on $S$.
			Then, in $T'$ there is a square with vertices $x,y,w,y'$, where $y'$ is the copy of $y$ in $T_2$.
			\item \emph{$\gamma_2(T') > 0$.}
			In fact, by counting the number of vertices and edges of $T'$ it is easy to conclude that $\gamma_2(T') = 2\cdot\gamma_2(T) > 0$.
		\end{itemize}
		Thus, $T'$ satisfies all the conditions in the statement, and by construction it has more vertices than $T$.
	\end{proof}
	
	\begin{remark}
		The construction in the proof of \Cref{prop:infinite_manysquare} does not produce new \minit{} triangulations, because $T'$ dominates $T$ with a map of degree $1$ sending the vertices in the interior of $T_1$ to $v$, and the vertices of $T_2$ to the corresponding ones in $T$.
	\end{remark}
	
	To search for new \minit{} triangulations distinct from $T_{10}$ and $T_{12}$, let us consider the following procedure:
	\begin{enumerate}
		\item Start from any flag triangulation of $S^3$;
		\item Perform a random sequence of edge subdivisions and edge collapses (when performing a collapse, the collapsed edge must not be contained in a square, otherwise flagness is lost);
		\item Keep collapsing edges not contained in squares, to obtain a candidate triangulation $T$ in which every edge is contained in a square;
		\item Check whether $\gamma_2(T) = 0$, or $T\gdom T_{10}$, or $T\gdom T_{12}$. If not, $T$ must dominate a \minit{} triangulation distinct from $T_{10}$ and $T_{12}$ (it could also be \minit{} itself).
	\end{enumerate}
	
	\begin{remark}
		The third step of the procedure isn't strictly necessary, but it helps in reducing the size of $T$, thus making the checks in the last step faster to perform.
		\Cref{rmk:squares_nolocal} below gives an additional reason why performing this step makes subsequent computations easier.
	\end{remark}
	
	Lutz and Nevo \cite{LN2016} proved that if $T$ and $T'$ are two flag PL-homeomorphic simplicial complexes, then it is possible to obtain $T'$ from $T$ by a finite sequence of edge subdivisions and their inverses (which are edge collapses), in such a way that the complex stays flag at every step.
	Moreover, the 3-dimensional \emph{Hauptvermutung} proved by Moise \cite{Moise1952} ensures that any two homeomorphic triangulated manifolds are PL homeomorphic.
	Thus, the procedure outlined above can theoretically produce any flag triangulation of $S^3$.
	If indeed there is a \minit{} triangulation $T$ distinct from $T_{10}$ and $T_{12}$, then there is an appropriate sequence of edge subdivisions and collapses that lead to $T$ in the second step of the procedure.
	
	While executing the third step, one might pass from a triangulation $T$ with $\gamma_2(T) > 0$ to a triangulation $T'$ with $\gamma_2(T') = 0$, obtained from $T$ by collapsing an edge.
	The following lemma ensures that, in this case, $T$ dominates $T_{10}$, so we are not missing \minit{} triangulations in this way.
	\begin{lemma}\label{lemma:becomes0}
		Let $T$ be a flag triangulation of $S^3$ containing an edge $e$ not contained in a square.
		Let $T'$ be the (flag) triangulation of $S^3$ obtained from $T$ by collapsing $e$.
		Suppose that $\gamma_2(T) > 0$ and $\gamma_2(T') = 0$.
		Then, $T\gdom T_{10}$.
	\end{lemma} 
	\begin{proof}
		Let $k$ be the number of vertices in the link of $e$ in $T$ (\ie, the number of vertices adjacent to both endpoints of $e$, but not in $e$).
		A quick computation gives $\gamma_2(T') = \gamma_2(T) - k + 4$.
		We have $\gamma_2(T') < \gamma_2(T')$, which implies $k \ge 5$.
		Now, the conclusion follows from \Cref{lemma:edgelink5}.
	\end{proof}
	
	\begin{remark}
		The proof of \Cref{lemma:becomes0} also shows that, while executing the third step of the procedure, the value of $\gamma_2$ never increases.
	\end{remark}
	
	The details of an actual implementation of the procedure outlined above can vary greatly. Most importantly:
	\begin{itemize}
		\item In the second step, there are many different choices for the number of operations to perform, and for the edges to apply them to.
		Different strategies can lead to exploring different ``regions'' of the space of flag triangulations;
		\item The implementation details of the fourth step are crucial on a practical standpoint: while computing $\gamma_2(T)$ does not pose any problems, since it is just a matter of counting vertices and edges, it is not clear how to check that $T \gdom T_{10}$ or $T \gdom T_{12}$ in an efficient way.
		Enumerating all simplicial maps is computationally very expensive, and it becomes factually impossible already when $T$ has around 20 vertices.
		In Subsection \ref{ssec:local} we discuss an approach that, in practice, allows to perform these checks, quite quickly, even on larger triangulations.
	\end{itemize}
	
	\begin{question}\label{q:dominance_complexity}
		How hard is the algorithmic problem of deciding, given two (flag) triangulations $S$ and $T$, whether $S \gdom T$?
		Does it admit a polynomial-time solution?
		What if the target triangulation $T$ is fixed (in particular, equal to $T_{10}$ or $T_{12}$), and just $S$ is part of the input?
		Note that the existence of a polynomial-time solution when both $S$ and $T$ are part of the input seems unlikely: see \Cref{q:polynomial_algo}, which is closely related.
	\end{question}

	\subsection{Local pictures and maps between them}\label{ssec:local}
	As already remarked, it is not clear how to check efficiently whether a triangulation dominates $T_{10}$ or $T_{12}$.
	Trying all simplicial maps takes too much time.
	In this subsection we describe a special type of simplicial maps, which behave in a specific way in the neighbourhood of a given edge.
	Restricting to such maps leads to much faster algorithms, that however (in principle) might fail to detect that a dominance relation holds, \ie, there might be \emph{false negatives}.
	
	We define a structure that we will use to encode the neighbourhood of an edge of a triangulation of $S^3$.
	\begin{definition}\label{def:localp}
		A \emph{local picture} $S$ is a flag triangulation of $S^2$ endowed with the following additional data:
		\begin{itemize}
			\item An \emph{equator} $E \subset S$, \ie, a full subcomplex homeomorphic to $S^1$. Note that $E$ has at least $4$ vertices, and divides $S$ into two hemispheres;
			\item A bipartite graph with nodes corresponding to the vertices of $S \setminus E$, partitioned into two sets corresponding to the two hemispheres of $S$;
			\item A subset of vertices of $S$ that we consider ``marked''.
		\end{itemize}
	\end{definition}
	
	\begin{definition}
		A map between two local pictures $S_1$ and $S_2$ is a simplicial map $f:S_1 \to S_2$ satisfying the following additional conditions:
		\begin{itemize}
			\item It restricts to a nonconstant and monotone map between the equators, of degree $\pm 1$;
			\item Let $H$ and $K$ be the two hemispheres of $S_1$.
			Then $f$ sends $H$ to one of the hemispheres of $S_2$ and $K$ to the other one;
			\item Adjacent nodes in the bipartite graph of $S_1$ are sent either to adjacent nodes in the bipartite graph of $S_2$, or to vertices which are adjacent (or equal) in the triangulation $S_2$;
			\item Marked vertices are sent to marked vertices.
		\end{itemize}
	\end{definition}
	
	Given a flag triangulation $T$ of $S^3$ and an edge $\{x,y\}\in T$, a local picture arises in the following way.
	Let $S_x$ and $S_y$ be the links of $x$ and $y$ in $T$, respectively. They are two flag triangulations of $S^2$.
	By removing $y$ (and its open star) from $S_x$, we obtain a flag triangulation of a $2$-disk, whose boundary coincides with the link of $\{x,y\}$ in $T$.
	In the same way, we obtain a disk with the same property by removing $x$ from $S_y$.
	By gluing these two disks along their common boundary, we obtain a flag triangulation $S$ of $S^2$, with an equator given by the link of $\{x,y\}$ and two hemispheres corresponding to the two disks.
	Note that $S$ is a subcomplex of $T$, but it might be nonfull, because vertices in the interior of a hemisphere could be adjacent (in $T$) to vertices in the other hemisphere: in fact this happens whenever $\{x,y\}$ is contained in squares in $T$.
	The bipartite graph in \Cref{def:localp} encodes these adjacency relations.
	Finally, a vertex of $S$ is marked if it is adjacent to some vertex which is adjacent neither to $x$ nor $y$.
	Figures \ref{fig:pT10}--\ref{fig:pT12_57} illustrate some examples of local pictures around some edges of $T_{10}$ and $T_{12}$.
	
	\begin{figure}[ht]
		\centering
		\includegraphics[scale=1]{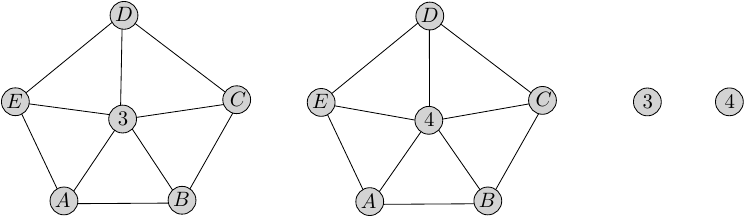}
		\caption{The local picture around edge $\{1,2\}$ of the triangulation $T_{10}$ defined as the join of two pentagons with vertex labels $\{A,B,C,D,E\}$ and $\{1,2,3,4,5\}$.
			The two hemispheres are drawn as separate disks.
			All the vertices are marked, and the bipartite graph, drawn separately on the right, has no edges.}
		\label{fig:pT10}
	\end{figure}
	\begin{figure}[ht]
		\centering
		\includegraphics[scale=1]{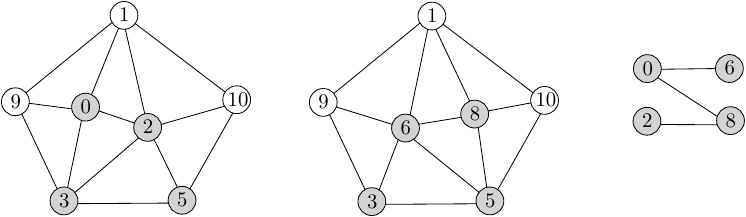}
		\caption{The local picture around edge $\{4,7\} \in T_{12}$.
		Vertices $0,2,3,5,6$ and $8$ are marked.}
		\label{fig:pT12_47}
	\end{figure}
	\begin{figure}[ht]
		\centering
		\includegraphics[scale=1]{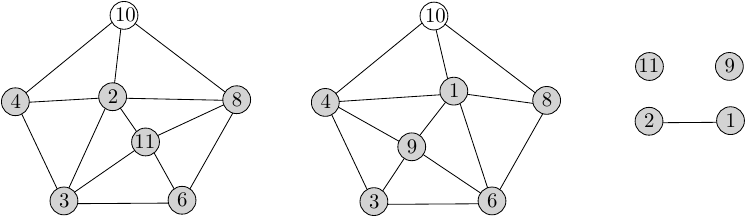}
		\caption{The local picture around edge $\{5,7\} \in T_{12}$.}
		\label{fig:pT12_57}
	\end{figure}
	
	\begin{definition}
		We say that an edge $\{x,y\}$ of a flag triangulation of $S^3$ is almost-omniscient if there is exactly one vertex that is adjacent neither to $x$ nor $y$.
	\end{definition}
	Thus, the local picture around an almost-omniscient edge contains all the vertices of the triangulation but three: the endpoints of the edge and the only vertex that is not adjacent to the edge.
	The local pictures of Figures \ref{fig:pT10}--\ref{fig:pT12_57} are all defined around almost-omniscient edges. 
	
	\Cref{lemma:loc_map} below explains how to recover a simplicial map between triangulations, starting from a map between local pictures, under the condition that the target local picture is defined around an almost-omniscient edge.
	\begin{lemma}\label{lemma:loc_map}
		Let $S$ be the local picture around an almost-omniscient edge of a flag triangulation $T$ of $S^3$.
		Suppose that $f:S' \to S$ is a map between local pictures, where $S'$ is the local picture around an edge of a flag triangulation $T'$.
		Then $f$ extends to a simplicial map from $T'$ to $T$ having degree $\pm 1$.
	\end{lemma}
	\begin{proof}
		Let $\{x,y\}\in T$ be the edge with local picture $S$, and $\{x',y'\}\in T'$ be the edge with local picture $S'$.
		Extend $f$ by sending $\{x',y'\}$ to $\{x,y\}$ in the way that is coherent with the way $f$ maps the hemispheres of $S'$ to the hemispheres of $S$; then, send the remaining vertices of $T'$ to the unique remaining vertex of $T$.
		
		From the fact that marked vertices are sent to marked vertices, it follows that this extension of $f$ is simplicial.
		Moreover, it has degree $\pm 1$ because for any tetrahedron of $T$ containing $\{x,y\}$ there is exactly one tetrahedron of $T'$ that is mapped onto it (this follows from the prescribed behaviour of $f$ on the equators).
	\end{proof}
	To find a map that certifies that a certain triangulation $T$ dominates $T_{10}$ or $T_{12}$, we can search for maps from local pictures in $T$ to the local pictures in Figures \ref{fig:pT10}--\ref{fig:pT12_57}.
	In practice, this strategy turns out to be successful, as explained in Subsection \ref{ssec:results}. 
	
	\begin{question}
		How hard is the algorithmic problem of deciding whether there is a map between two local pictures?
		Does it admit a polynomial-time solution?
		The algorithm implemented by the author for the searches described in Subsection \ref{ssec:results} builds a map vertex by vertex, making sure that all conditions are satisfied, until it is defined on all vertices or a dead end is reached, in which case it reverts the last choices until the whole tree of possibilities is traversed.
		This works reasonably well in practice, despite having an exponential running time.
	\end{question}
	
	\subsection{More details on the search and its results}\label{ssec:results}
	The searches performed by implementing on a computer the procedure outlined in Subsection \ref{ssec:quest} didn't lead to \minit{} triangulations distinct from $T_{10}$ and $T_{12}$.
	In this subsection, we discuss some more details about the specific implementation by the author, which is available in a GitHub repository \cite{simplicial-spheres}, and the results obtained by running the resulting computer program.
	The structure of the implemented search algorithm is outlined below.
	Notice how two sequences of triangulations appear, denoted by $T_i$ and $T_i'$.
	The idea is that we perform randomly a ``walk'' in the space of triangulations, and the visited triangulations are the ones denoted by $T_i$.
	The triangulations that are actually analysed are the ones with the prime symbol, which are the ``simplified'' triangulations obtained by collapsing edges not contained in squares.
	\begin{itemize}
		\item Start from a flag triangulation $T_0$ (\ie, the join of two pentagons);
		\item For $i = 1,\dots,N$, repeat the following series of steps:
		\begin{enumerate}
		\item \label{it:sub} Obtain a new triangulation $T_{i}$ from $T_{i-1}$ by performing some edge subdivisions;
		\item \label{it:con} Modify $T_{i}$ by performing some edge collapses;
		\item \label{it:squ} Obtain a new triangulation $T_i'$ from $T_i$ by collapsing edges, until every edge of $T_i'$ is contained in a square ($T_i$ is preserved for the next iteration, where it is needed to construct $T_{i+1}$);
		\item \label{it:check} If $\gamma_2(T_i') > 0$, for every edge of $T_i'$, construct its local picture and search for a map to the local pictures of Figures \ref{fig:pT10} -- \ref{fig:pT12_57} (which come from $T_{10}$ and $T_{12}$).
		\end{enumerate}
	\end{itemize}
	
	In Step \ref{it:squ} (the collapse of edges not contained in squares), the edge to be collapsed is selected uniformly at random among those not contained in squares, and this process is repeated until every edge is contained in a square.
	More sophisticated strategies could be employed, like collapsing edges whose link has more (or less) vertices than others.
	In fact, as explained below, this strategy is used in Step \ref{it:con}.
	
	
	The number $K$ of subdivisions performed in Step \ref{it:sub} is fixed beforehand.
	Then, in Step \ref{it:con}, the algorithm \emph{attempts} to perform a number of collapses $K'$, which is approximately equal to $K$ but can vary slightly in different iterations, according to the number of vertices of $T'_{i-1}$:
	if the number of vertices is bigger than a fixed threshold, then $K'$ is slightly bigger than $K$; if it is smaller than another threshold, then $K'$ is slightly smaller than $K$; if it lies in the interval between the two thresholds, then $K' = K$.
	This encourages the number of vertices of $T'_i$ to stay in a fixed window.
	Note that the execution of Step \ref{it:con} could end prematurely because the edges that can be collapsed (those not contained in squares) might run out.
	
	As for the selection of the edge to be subdivided or collapsed, in Step \ref{it:sub} (for subdivisions) it is selected uniformly at random among all edges.
	In Step \ref{it:con}, instead, two behaviours are implemented:
	\begin{itemize}
		\item If $\gamma_2(T_{i-1}')$ is smaller than a fixed threshold, then the edge is chosen uniformly at random;
		\item If $\gamma_2(T_{i-1}')$ is bigger than the threshold, then the edge is chosen at random among those having the biggest links.
	\end{itemize}
	The rationale behind this strategy is trying to keep $\gamma_2$ from growing too much iteration after iteration (see the proof of \Cref{lemma:becomes0}) and to avoid, at the same time, that $\gamma_2(T_i')=0$ too often.
	\begin{remark}
		It appears, from the runs performed by the author, that if in Step \ref{it:con} the edge is always chosen at random, without considering the link sizes, then not only $\gamma_2$ tends to grow, but also the number of vertices of $T_i$ grows, due to the edges not contained in squares running out prematurely.
	\end{remark}

	The algorithm has been run several times, adjusting the values of the parameters described above in different ways across the various runs.
	For every $T_i'$ with positive $\gamma_2$, a map to a local picture coming from $T_{10}$ or $T_{12}$ was found, certifying that $T_i'\gdom T_{10}$ or $T_i'\gdom T_{12}$ (or both).
	
	The following table is obtained by aggregating the results of all the runs.
	%
	%
	The table lists, for every $v$, how many \emph{pairwise nonisomorphic} triangulations $T_i'$ having exactly $v$ vertices have been ``visited'' during these runs (that together correspond to a total of roughly $N = 10\,000\,000$ iterations of the procedure).
	Recall that all these triangulations have the following properties: they are flag and their edges are contained in squares.
	
	\vspace{.7em}
	\noindent\begin{tabular}{c|c|c|c|c|c|c|c|c|c}
		\bfseries{N. of vertices} & 8 & 12 & 13 & 14 & 15 & 16 & 17 & 18 &  \dots\\
		\hline
		& 1 & 1  & 1  & 2  & 10 & 99 & 1\,051 & 17\,317 & \dots
	\end{tabular}
	\vspace{.7em}\\
	\begin{tabular}{c|c|c|c|c|c|c|c}
		\dots & 19 & 20 & 21 & 22 & 23 & 24 & \dots\\
		\hline
		\dots & 197\,874 & 503\,708 & 573\,314 & 549\,136 & 495\,291 & 427\,299 & \dots
	\end{tabular}
	\vspace{.7em}\\
	\begin{tabular}{c|c|c|c|c|c|c|c}
		\dots & 25 & 26 & 27 & 28 & 29 & 30 & \dots\\
		\hline
		\dots & 353\,892 & 278\,300 & 209\,255 & 151\,772 & 106\,825 & 73\,043 & \dots
	\end{tabular}
	\vspace{.7em}\\
	\begin{tabular}{c|c|c|c|c|c|c|c|c}
		\dots & 31 & 32 & 33 & 34 & 35 & 36 & 37 & \dots \\
		\hline
		\dots & 48\,426 & 31\,656 & 20\,546 & 12\,994 & 8\,398 & 5\,223 & 3\,472  & \dots
	\end{tabular}
	\vspace{.7em}\\
	\begin{tabular}{c|c|c|c|c|c|c|c|c|c|c}
		\dots & 38 & 39 & 40 & 41 & 42 & 43 & 44 & 45 & 46 & \dots \\
		\hline
		\dots & 2\,169 & 1\,425 & 877 & 593 & 413 & 304 & 183 & 122 & 83 
	\end{tabular}
	\vspace{.7em}\\
	\begin{tabular}{c|c|c|c|c|c|c|}
		\dots & 47 & 48 & 49 & 50 & 51 & 52\\
		\hline
		\dots & 62 & 23 & 17 & 4 & 1 & 1
	\end{tabular}
	\vspace{.7em}

	\begin{remark}
    	The procedure finds exactly one triangulation with 8 vertices: the octahedral triangulation.
		It is also the only encountered triangulation with $\gamma_2 = 0$.
		This is evidence in support of a conjecture of Lutz and Nevo \cite[Conjecture 6.1]{LN2016}, see \Cref{conj:collapses}.
	\end{remark}
	\begin{remark}
		There is a ``gap'' between 8 and 12 vertices.
		In fact, it can be proved that any flag triangulation with 9, 10 or 11 vertices has ``collapsible'' edges (not contained in squares).
		The triangulation with 12 vertices is $T_{12}$.
	\end{remark}
	\begin{conjecture}
		$T_{12}$ is, up to isomorphism, the only flag triangulation of $S^3$ with 12 vertices in which every edge is contained in a square.
	\end{conjecture}
	It seems plausible, looking at the result of the executions, that all the relevant triangulations with $\le 16$ vertices have been found.
	In fact, they have been all ``discovered'' in the first few thousands of iterations (and then ``rediscovered'' many times afterwards).
	Also, the plots in \Cref{fig:v17-18} might suggest that most of the triangulations with 17 vertices, and a very good portion of those with 18 vertices, have been analysed as well. 
	Compare \Cref{fig:v20}, which instead indicates that only a small percentage of the triangulations with 20 vertices has been found and checked.
	\begin{figure}[ht]
		\centering
		\begin{minipage}{.47\textwidth}
			\centering
			\includegraphics[scale=.36]{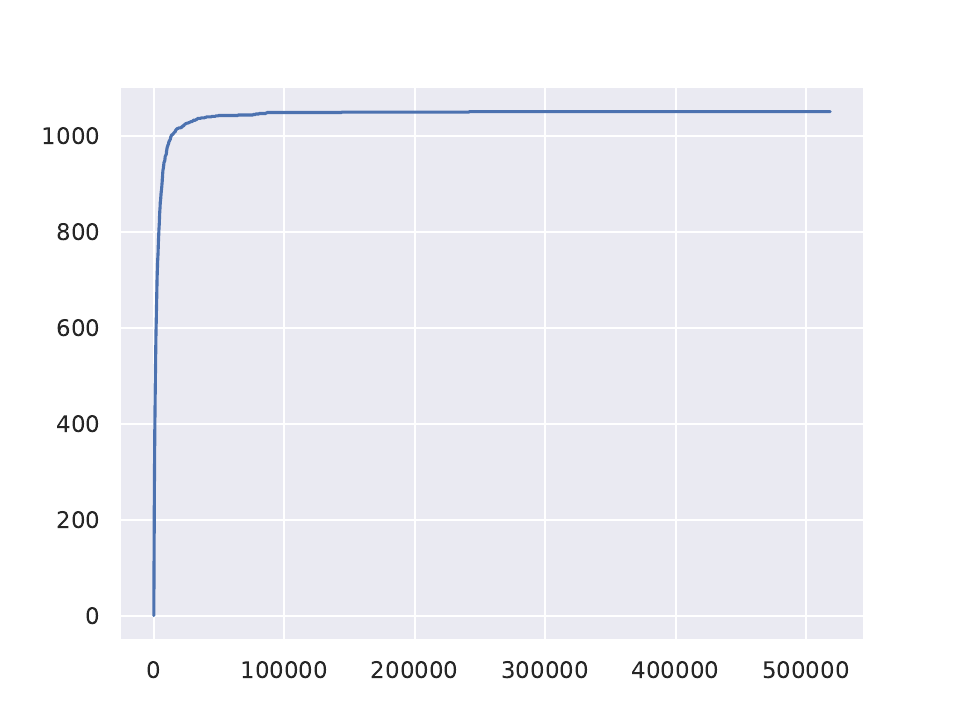}
		\end{minipage}%
		\hspace{.05\textwidth}
		\begin{minipage}{.47\textwidth}
			\centering
			\includegraphics[scale=.36]{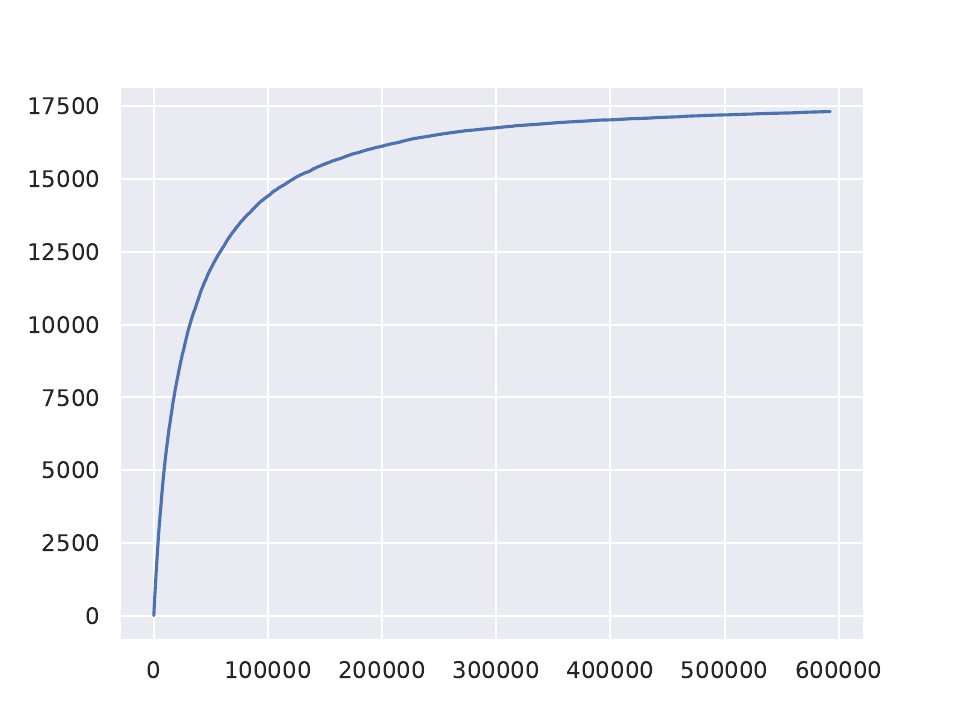}
		\end{minipage}
		\caption{The growth of the number of distinct triangulations with a fixed number of vertices (17 on the left, 18 on the right).
		A unit on the horizontal axis corresponds to the event of finding such a triangulation during the execution of the algorithm; there is a corresponding increment on the vertical axis when the triangulation is not isomorphic to the ones discovered up to that point.
		The plot becoming flatter and flatter means that finding a novel triangulation becomes a quite rare event.}
		\label{fig:v17-18}
	\end{figure}
	\begin{figure}[!ht]
		\centering
		\includegraphics[scale=.36]{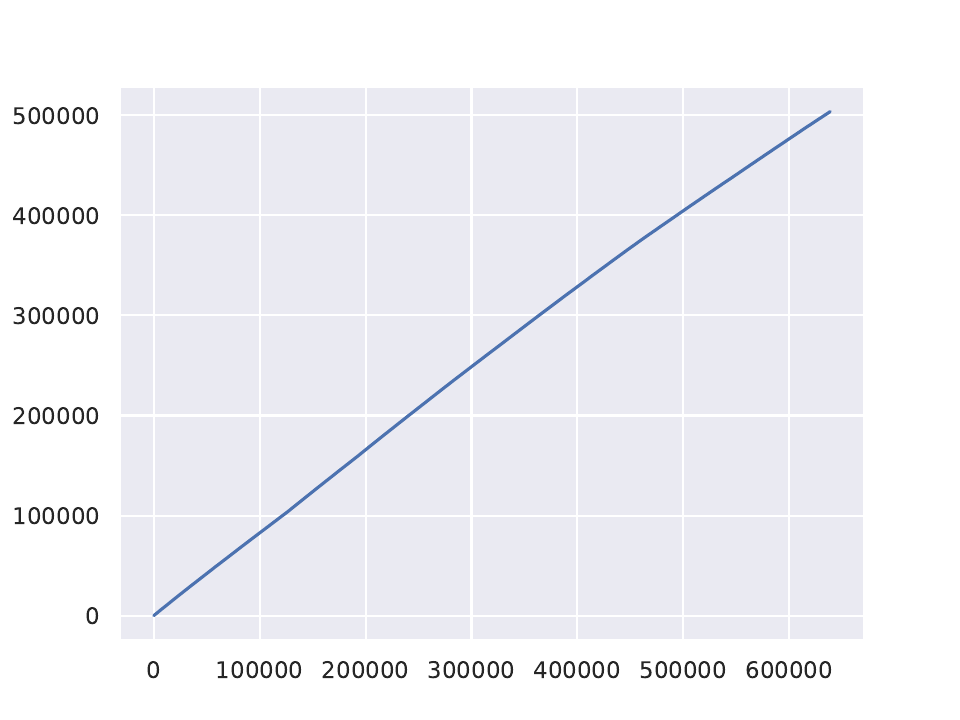}
		\caption{The growth of the number of distinct triangulations with 20 vertices, while executing the algorithm.
		Notice how the plot is approximately linear, with slope close to $1$.}
		\label{fig:v20}
	\end{figure}
	\begin{remark}\label{rmk:squares_nolocal}
		Since $T_i \gdom T_i'$, and every encountered $T_i'$ dominates either $T_{10}$ or $T_{12}$ via a map between local pictures, by transitivity we have $T_i \gdom T_{10}$ or $T_i \gdom T_{12}$.
		There are cases in which there are \emph{no} maps from local pictures of $T_i$ to local pictures of $T_{10}$ or $T_{12}$.
		Thus, the existence of a map between local pictures is a stronger property than dominance, in general.
		This makes Step \ref{it:squ} of the algorithm very helpful, not only because it reduces the size of the triangulation, but most importantly beacuse the ``nonsimplified'' triangulation $T_i$ might not dominate $T_{10}$ or $T_{12}$ via maps between local pictures, while the ``simplified'' $T_i'$ does.
	\end{remark}

	\section{Dominance vs graph minors}\label{sec:minors}
	In this section we establish a connection between the notion of dominance between triangulations (\Cref{def:dominance}) and the notion of minors from graph theory.
	Once such a connection is established, powerful theorems from the theory of graph minors allow us to obtain many consequences about the dominance relation.
	Unfortunately, we are able to state a precise connection only for triangulations of $S^2$, a case that we already understand pretty well with the results in \Cref{sec:3dim}.
	However, some of the consequences are interesting even in this particular case.
	
	Let us recall the definition of \emph{minors} of a graph.
	\begin{definition}
		Let $G$ and $H$ be two finite graphs.
		$H$ is said to be a minor of $G$ if it is isomorphic to a graph obtained from $G$ by applying a finite sequence of operations of the following types:
		\begin{itemize}
			\item Deletion of an edge;
			\item Collapse of an edge (two adjacent vertices are identified, and one edge joining them is removed);
			\item Deletion of an isolated vertex.
		\end{itemize}
	\end{definition}
	The third type of operation is almost superfluous, in the sense that it can be easily accomplished by a sequence of operations of the first two types, unless $G$ has some connected components that have to be completely deleted.
	In particular, if $G$ is connected (and $H$ is nonempty) we can always avoid this operation.
	
	Also note that we have to allow graphs to have multiple edges between two vertices, as well as loops (\ie{}, edges whose two endpoints coincide) if we want the notion to be stable under collapses of edges.
	However, one can always modify the sequence of operations and perform all deletions of edges before the collapses; by doing so, if both $G$ and $H$ are simplicial graphs, the graph stays simplicial at all times.
	
	The following theorem is considered one of the most remarkable in graph theory; it is one of the main results in a series of more than twenty papers by Robertson and Seymour, in which they develop a very deep theory around the notion of graph minors.
	\begin{theorem}[\cite{RS2004}]\label{thm:minors}
		Let $G_1, G_2, \dots$ be an infinite sequence of finite graphs.
		Then there are indices $i < j$ such that $G_i$ is a minor of $G_j$.
	\end{theorem}
	
	With the following proposition we connect the notion of dominance (\Cref{def:dominance}) with that of graph minors, at least for triangulations of $S^2$.
	\begin{proposition}\label{prop:minor_s2}
		Let $T_1$ and $T_2$ be simplicial complexes homeomorphic to $S^2$, and denote by $G_1$ and $G_2$ their $1$-skeletons.
		If $G_1$ is a minor of $G_2$, then there is a simplicial map $f:T_2\to T_1$ of degree $\pm 1$.
	\end{proposition}
	\begin{proof}
		Consider a sequence of operations that realises $G_1$ as a minor of $G_2$; we can assume that only collapses and deletions of edges are performed.
		By keeping track of the ``history'' of each vertex of $G_2$ during these operations (recall that we never delete vertices), we obtain a map (we think of it as a labelling of the vertices of $G_2$) $f:V(G_2) \to V(G_1)$ with the following properties:
		\begin{itemize}
			\item For every $v \in V(G_1)$, its preimage $f^{-1}(\{v\})$ spans a connected subgraph of $G_2$, \ie{}, any two vertices in $G_2$ sharing the same label are the endpoints of a path in $G_2$ visiting only vertices with that same label;
			\item If $v,w\in V(G_1)$ are adjacent, then there are $v',w' \in V(G_2)$ that are adjacent and have labels $f(v') = v$ and $f(w') = w$.
		\end{itemize}
		For every $v \in V(G_1)$, we fix a spanning tree of $f^{-1}(\{v\})$, \ie{}, a subgraph of $G_2$ without cycles whose vertices are exactly those labelled with $v$.
		We say that an edge which is part of this subgraph is $v$-preferred.
		We also fix, for every pair of adjacent vertices $v,w \in V(G_1)$, an edge in $G_2$ connecting two vertices labelled with $v$ and $w$.
		We say that this edge is $\{v,w\}$-preferred.
		
		We say that an edge of $G_2$ is preferred if it is $v$-preferred for some $v$ or $\{v,w\}$-preferred for some $v,w$.
		Let $P$ be the subgraph of $G_2$ with $V(P) = V(G_2)$ and whose edges are exactly the preferred ones.
		The graph $P$ divides the sphere $T_2$ in a number of connected components which is equal to the number of $2$-simplices of $T_1$.
		This follows, for instance, from the following fact: the number of components that a planar graph $G$ cuts the plane (or the sphere) into is equal to $2-V(G)+E(G)$; and it is easy to see that $V(P)-E(P) = V(T_1)-E(T_1)$.
		
		Let $\sigma=\{v_1, v_2,v_3\}$ be a $2$-simplex in $T_1$.
		There is a unique simple closed path $\gamma_\sigma$ in $P\subset T_2$ passing only through vertices labelled with $v_1$, $v_2$ and $v_3$.
		It cuts $T_2$ into two components homeomorphic to discs.
		If $x,y \in V(T_1) = V(G_1)$ are vertices not belonging to $\sigma$, then they are joined by a path in $G_1$ that avoids $v_1, v_2$ and $v_3$.
		This implies that the set of vertices in $G_2$ labelled with $x$ or $y$ are all contained in the same component of $T_2\setminus \gamma_\sigma$.
		Thus, among the two discs that $\gamma_\sigma$ cuts $T_2$ into, there is an ``internal'' one that only (possibly) contains vertices labelled with $v_1,v_2$ and $v_3$, and an ``external'' one in which the vertices with any other label must lie.
		We denote by $D_\sigma$ the internal disc; it is a subcomplex of $T_2$.
		The boundary of $D_\sigma$ is contained in $P$ by construction; $P$ may also intersect the interior of $D_\sigma$, but it cannot split it in multiple connected components.
		
		We have just described a way to associate to any $2$-simplex $\sigma$ of $T_1$ a connected component of $T_2\setminus P$, which is given by $D_\sigma \setminus P$.
		This association is injective, because $\sigma$ is determined by the set of labels of vertices belonging to the (closure of the) associated connected component.
		But then it must be bijective, being a map between finite sets of the same cardinality.
		This implies that the discs $D_\sigma$, as $\sigma$ varies among the $2$-simplices of $T_1$, cover $T_2$; in particular, every simplex of $T_2$ is contained in some $D_\sigma$, and therefore the labels of its vertices are all contained in a simplex of $T_1$.
		This means that $f$ defines a simplicial map $f:T_2\to T_1$.
		
		Let us show that $\abs{\deg(f)} = 1$.
		Pick any $2$-simplex $\sigma$ in $T_1$.
		By restricting $f$ we get simplicial maps $f_\sigma:D_\sigma \to \sigma$ and $f_{\partial\sigma}:\partial D_\sigma \to \partial \sigma$.
		It is clear that $f_{\partial_\sigma}$ has degree $\pm 1$; therefore, also $f_\sigma$ has degree $\pm 1$.
		This means that, counting with sign the $2$-simplices in $D_\sigma$ that are mapped surjectively onto $\sigma$, we get $\pm 1$.
		On the other hand, if a $2$-simplex of $T_2$ is not contained in $D_\sigma$, then it cannot be mapped surjectively onto $\sigma$, because it is contained in some other $D_{\sigma'}$.
		This shows that $\abs{\deg(f)} = 1$.
	\end{proof}
	
	\begin{corollary}\label{cor:well_dominance}
		Any set of pairwise-nonisomorphic triangulations of $S^2$, endowed with the partial order given by dominance (\Cref{def:dominance}), has a finite number of minimal elements.
	\end{corollary}
	\begin{proof}
		If this is not the case, enumerate the minimal triangulations $T_1, T_2, \dots$ and consider their $1$-skeletons.
		By \Cref{thm:minors}, the $1$-skeleton of $T_i$ is a minor of the $1$-skeleton of $T_j$ for some indices $i < j$.
		By \Cref{prop:minor_s2}, $T_j$ dominates $T_i$, against the minimality of $T_j$.
	\end{proof}
	In particular, we immediately rediscover the following nonconstructive and less precise version of \Cref{thm:2dim}.
	\begin{corollary}
		The poset of triangulations of $S^2$ giving rise to manifolds with positive simplicial volume, with the partial order given by dominance (\Cref{def:dominance}), has a finite number of minimal elements, up to isomorphism.
	\end{corollary}
	
	\begin{question}
		Exploiting \Cref{thm:minors}, or some variations of it, can we obtain the same conclusion (finiteness of the set of minimal triangulations, among those giving positive simplicial volume, or those with $\gamma_2 > 0$) for triangulations of $S^3$?
	\end{question}
	Here is another corollary.
	We already know the case $T = T_9$, that holds with $d_T = 1$ by \Cref{thm:2dim}.
	\begin{corollary}\label{cor:dim2_deg}
		Let $T$ be a triangulation of $S^2$.
		There is a natural number $d_T$ such that, if $S$ is any triangulation of $S^2$ that dominates $T$, then it does so via a simplicial map $f$ with $\abs{\deg(f)} \le d_T$.
	\end{corollary}
	\begin{proof}
		Let $\mathcal{T}$ be the set of triangulations of $S^2$ (up to isomorphism) that dominate $T$.
		Endow $\mathcal{T}$ with the following partial order: $S_1 \le S_2$ if the $1$-skeleton of $S_1$ is a minor of the $1$-skeleton of $S_2$.
		By \Cref{thm:minors}, the poset $\mathcal{T}$ has a finite number of minimal elements, that we call $T_1, \dots, T_k$.
		Let $d_T$ be such that any $T_i$ dominates $T$ via a map of degree at most $d_T$ in absolute value. 
		
		Take now $S \in \mathcal{T}$.
		There must be a $T_i$ such that the $1$-skeleton of $T_i$ is a minor of the $1$-skeleton of $S$.
		By \Cref{prop:minor_s2}, $S$ dominates $T_i$ via a map of degree $\pm 1$, and so, by composition, it dominates $T$ via a map of degree at most $d_T$ in absolute value.
	\end{proof}
	\begin{question}
		Is there a more direct proof of this?
		Can we take $d_T = 1$?
		What about higher dimensions?
	\end{question}

	\subsection{Algorithmic considerations}
	We recall another result from the series of papers by Robertson and Seymour, about the algorithmic complexity of deciding whether a graph is a minor of another.
	\begin{theorem}[{\cite{RS1995}}]\label{thm:minor_alg}
		Let $H$ be a finite graph.
		There is an algorithm that takes as input a finite graph $G$ and decides whether $H$ is a minor of $G$ in time $O(\abs{V(G)}^3)$, where $V(G)$ is the set of vertices of $G$.
	\end{theorem}
	A technical note about the time complexity in \Cref{thm:minor_alg}, in case of the presence of multiple edges between vertex pairs: we assume that, given two vertices $v$ and $w$ of $G$ and an integer $k$, we can decide whether there are (at least) $k$ edges with endpoints $v$ and $w$ in a time that depends only on $k$.
	
	\begin{corollary}\label{cor:dim2_alg}
		Let $T$ be a triangulation of $S^2$.
		There is an algorithm that takes as input a triangulation $S$ of $S^2$ and decides in time $O(\abs{V(S)}^3)$ whether $S$ dominates $T$.
	\end{corollary}
	\begin{proof}
		Let $\mathcal{T}$ and $T_1, \dots, T_k$ be as in the proof of \Cref{cor:dim2_deg}.
		Let $S$ be a triangulation of $S^2$.
		Note that, if $S \in \mathcal{T}$, then there is some $T_i$ whose $1$-skeleton is a minor of the $1$-skeleton of $S$.
		On the other hand, if there is a $T_i$ whose $1$-skeleton is a minor of the $1$-skeleton of $S$, then $S$ dominates $T_i$ by \Cref{prop:minor_s2}, and thus $S \in \mathcal{T}$.
		The conclusion follows by applying \Cref{thm:minor_alg} with $H = T_i$, for each $T_i$. 
	\end{proof}
	
	\begin{question}\label{q:polynomial_algo}
		The proof of \Cref{cor:dim2_alg} is nonconstructive, because we don't have a procedure to find the triangulations $T_1, \dots, T_k$, starting from $T$.
		Is there a more explicit polynomial-time algorithm for the same problem?
		What if we let $T$ be part of the input?
		In other words, is there a universal algorithm that works for every $T$?
		Note that the algorithmic problem of deciding whether a graph is a minor of another is known to be NP-complete, so it is very unlikely that it has a polynomial-time solution. 
	\end{question}

	\section{Other questions}\label{sec:questions}
	It would be good to understand the simplicial volume of manifolds associated to triangulations with $\gamma_2=0$, as well.
	Up to now I am not aware of examples of such manifolds with positive simplicial volume.
	\begin{question}\label{q:converse}
		Is there a flag triangulation $T$ of $S^3$ with $\gamma_2(T) = 0$ and $\SV{\Dav{T}} > 0$?
		Or, even outside the context of Davis' construction, is there a closed aspherical four-dimensional manifold with vanishing Euler characteristic and positive simplicial volume?
	\end{question}
	The same questions in even dimensions strictly larger than $4$ have positive answers, since we can take the product of two odd-dimensional manifolds with positive simplicial volume (the odd-dimensional case is obvious).
	
	\begin{question}\label{q:collapse}
		Let $T$ be a flag triangulation of $S^3$ and $e$ be a $1$-simplex in $T$ which is not contained in any square and whose link is a square.
		Let $T'$ be the result of collapsing $e$; it is another flag triangulation of $S^3$.
		If $\Dav{T}$ has positive simplicial volume, does $\Dav{T'}$ have positive simplicial volume?
	\end{question}
	It is readily seen that $\gamma_2(T') = \gamma_2(T)$.
	Therefore, a negative example to to \Cref{q:collapse} would either produce a counterexample to \Cref{conj:sv_chi}, if $\gamma_2(T) > 0$, or answer positively to \Cref{q:converse}, if $\gamma_2(T) = 0$.
	This is also related to the following conjecture of Lutz and Nevo.
	\begin{conjecture}[The three-dimensional case of {\cite[Conjecture 6.1]{LN2016}}]\label{conj:collapses}
		Let $T$ be a flag triangulation of $S^3$ with $\gamma_2(T) = 0$.
		Then there is a sequence of collapses as in \Cref{q:collapse} that reduce $T$ to the octahedral $3$-sphere.
	\end{conjecture}
	If \Cref{conj:collapses} were true, a positive answer to \Cref{q:converse} would give a negative answer to \Cref{q:collapse}.
	It would also imply that, if in a flag triangulation $T$ of $S^3$ every edge is contained in a square, then either $\gamma_2(T) > 0$ or $T$ is isomorphic to the octahedral $3$-sphere.
	
	The following question is related to some results and questions already considered in \Cref{sec:minors}.
	Recall that in \Cref{def:dominance}, where the dominance relation is introduced, we require the existence of a simplicial map of any nonzero degree.
	But in all cases in which we have concretely produced these maps, they all were of degree $\pm 1$.
	\begin{question}
		Let $T$ and $T'$ be two flag triangulations of $S^n$, and suppose that $T \gdom T'$.
		Is there a simplicial map from $T$ to $T'$ of degree $\pm 1$?
	\end{question}
	\Cref{cor:dim2_deg} gives a partial answer in the case $n = 2$.
	\begin{question}
		Let $T$ and $T'$ be two flag triangulations of $S^3$, and suppose that $T \ldom T'$.
		Is it true that $\gamma_2(T) \le \gamma_2(T')$?
	\end{question}
	
	Switching to another line of thoughts, we report the following result of Genevois.
	\begin{proposition}[{\cite{Genevois22}}]
		Let $W$ be a simplicial graph and $u \in V(W)$ a vertex.
		Let $W_u$ denote the graph obtained by gluing two copies of $W\setminus\{u\}$ along the link of $u$.
		Then the right-angled Coxeter group associated to $W_u$ is isomorphic to a subgroup of index $2$ of the one associated to $W$.
	\end{proposition}
	In particular, we deduce the following.
	\begin{corollary}
		Let $T$ be a flag triangulation of $S^3$ and $u \in V(T)$ be a vertex.
		Let $T_u$ be the simplicial complex obtained by gluing two copies of $T\setminus\{u\}$ along the link of $u$.
		Then $T_u$ is a flag triangulation of $S^3$, and $\SV{\Dav{T_u}}>0$ if and only if $\SV{\Dav{T}}>0$.
	\end{corollary}
	\begin{question}
		How does this relate to the dominance relation?
		For example, if $T_u$ dominates some other triangulation (say, $T_{10}$), is the same true for $T$?
	\end{question}
	If $T_u$ dominates $T_{10}$ and $T$ does not, we would conclude that $\Dav{T}$ has positive simplicial volume, and then deduce that some \minit{} triangulations distinct from $T_{10}$ give positive simplicial volume.
	
	\bibliographystyle{fram_plain}
	\bibliography{bibliography.bib}
\end{document}